\documentclass{amsart}
\usepackage{amssymb}
\usepackage{comment}
\usepackage{tikz}
\usetikzlibrary{matrix}
\usetikzlibrary{shapes}
\usetikzlibrary{decorations.pathreplacing, arrows.meta}
\usepackage{enumerate}

\newcommand{\uhr}{\upharpoonright}
\newcommand{\concat}{^\smallfrown}
\tikzstyle{block} = [draw,minimum size=1em]
\newcommand{\RCA}{\mathsf{RCA}_0}
\newcommand{\ACA}{\mathsf{ACA}_0}
\newcommand{\WKL}{\mathsf{WKL}_0}
\newcommand{\ATR}{\mathsf{ATR}_0}
\newcommand{\DPB}{\mathsf{CD}\text{-}\mathsf{PB}}
\newcommand{\JI}{\mathsf{JI}}
\newcommand{\LwCA}{\mathsf{L_{\omega_1,\omega}}\text{-}\mathsf{CA}}
\newcommand{\DCA}{\Delta^1_1\text{-}\mathsf{CA_0}}
\newcommand{\BorelDRT}{\mathsf{Borel}\text{-}\mathsf{DRT}}
\newcommand{\ceil}[1]{\lceil #1 \rceil}

\newcommand{\makeset}[1]{|#1|}
\newcommand{\Decorate}{\operatorname{Decorate}}

\newcommand{\kO}{\mathcal O}

\newtheorem{theorem}{Theorem}[section]
\newtheorem{definition}[theorem]{Definition}
\newtheorem{proposition}[theorem]{Proposition}
\newtheorem{lemma}[theorem]{Lemma}

\newtheorem{question}[theorem]{Question}
\newtheorem*{restatement}{Restatement}

\title{The determined property of Baire in reverse math}

\author[Astor]{Eric P. Astor}
\address{Google LLC\\
111 8th Ave.\\
New York, NY 10011, U.S.A.}
\email{eric.astor@gmail.com}

\author[Dzhafarov]{Damir Dzhafarov}
\address{Department of Mathematics\\
University of Connecticut\\
Storrs, Connecticut U.S.A.}
\email{damir.dzhafarov@uconn.edu}

\author[Montalb\'an]{Antonio Montalb\'an}
\address{Department of Mathematics\\
University California-Berkeley\\
Berkeley, California U.S.A.}
\email{antonio@math.berkeley.edu}

\author[Solomon]{Reed Solomon}
\address{Department of Mathematics\\
University of Connecticut\\
Storrs, Connecticut U.S.A.}
\email{solomon@math.uconn.edu}

\author[Westrick]{Linda Brown Westrick}
\address{Department of Mathematics\\
Penn State University\\
University Park, Pennsylvania U.S.A.}
\email{westrick@psu.edu}

\thanks{Dzhafarov was supported by grant DMS-1400267 from the National Science Foundation of the United States and a Collaboration Grant for Mathematicians from the Simons Foundation.}

\begin{document}

\begin{abstract}
We define the notion of a completely determined Borel code in reverse 
mathematics, and consider the principle $\DPB$, which states 
that every completely determined Borel set has the property of Baire. 
We show that this principle is strictly weaker than 
$\ATR$. Any $\omega$-model of $\DPB$ must be closed under 
hyperarithmetic reduction, but $\DPB$ is not a theory 
of hyperarithmetic analysis.  We show that whenever 
$M\subseteq 2^\omega$ is the second-order part of an 
$\omega$-model of $\DPB$, then for every $Z \in M$,
there is a $G \in M$ such that $G$ is $\Delta^1_1$-generic 
relative to $Z$.
\end{abstract}

\maketitle

\section{Introduction}

The program of reverse mathematics aims to quantify the 
strength of the various axioms and 
theorems of ordinary mathematics by assuming only a 
weak base theory ($\RCA$) and then determining which 
axioms and theorems can prove which others over that
weak base.  Five robust systems emerged, 
(in order of strength,
$\RCA, \WKL, \ACA, \ATR, \Pi^1_1\text{-}\mathsf{CA}_0$)
with most 
theorems of ordinary mathematics being equivalent 
to one of these five (earning this group the moniker ``the big five'').  
The standard reference is 
\cite{sosa}.  In recent decades, most work in reverse mathematics
has focused on the theorems that 
do not belong to the big five but are
in the vicinity of $\ACA$.  
Here we discuss two principles which are 
outside of the big five and
located in the general vicinity of $\ATR$:
the \emph{property of Baire for completely determined Borel sets}
($\DPB$) and the \emph{Borel dual Ramsey theorem for 3 partitions 
and $\ell$ colors} ($\BorelDRT^3_\ell$). Both 
principles involve Borel sets.  

Our motivation is to make it possible
to give a meaningful reverse mathematics analysis of theorems
whose statements involve Borel sets.
The way that Borel sets are usually 
defined in reverse mathematics forces many theorems that even mention 
a Borel set to imply $\ATR$, in an unsatisfactory 
sense made precise in 
\cite{DFSW}.  Here we propose another
definition for a Borel set in reverse mathematics, distinguished
from the original by the terminology 
\emph{completely determined Borel set}, and to put bounds on the 
strength of the statement
$$\DPB:  \text{``Every completely determined Borel set has the property of Baire''}$$
This statement should be compared with the usual
``Every Borel set has the property of Baire'', which \cite{DFSW} 
showed is 
equivalent to $\ATR$ for aforementioned empty reasons.
In contrast, working with $\DPB$ requires working with 
hyperarithmetic generics, giving this theorem more thematic 
content.  While we do not claim that $\DPB$ is the ``right'' 
formalization of the principle that Borel sets have the 
Baire property, it is a step in that direction.

We show that over $\RCA$, $\DPB$ is implied by $\ATR$ and
implies $\LwCA$. Our first main theorems say that both
implications are strict.
\begin{theorem}\label{thm:01}
  There is an $\omega$-model of $\DPB$ in which $\ATR$ fails.
\end{theorem}
\begin{theorem}\label{thm:01a}
There is an $\omega$-model of $\LwCA$ in which $\DPB$ fails. 
In fact, $HYP$ is such an $\omega$-model.
\end{theorem}

This establishes that $\DPB$ is located in the general 
vicinity of the
\emph{theories of hyperarithmetic analysis},
a mostly linearly ordered collection of logical 
principles which are strong enough to support 
hyperarithmetic reduction, but too weak to imply 
the existence of jump hierarchies.  With the 
exception of Jullien's indecomposability theorem 
\cite{Montalban2006}, no
theorems of ordinary mathematics are known to exist in 
this space.  The only known statement of 
hyperarithmetic analysis that is not linearly ordered 
with the others is the arithmetic Bolzano-Weierstrass 
theorem (see \cite{Friedman1975}, \cite{Conidis2012}).
Now, $\DPB$ is not a theory of hyperarithmetic 
analysis because it does not hold in $HYP$. However these 
theories of hyperarithmetic analysis are 
the closest principles to $\DPB$ that have already 
been studied.

  \begin{center}
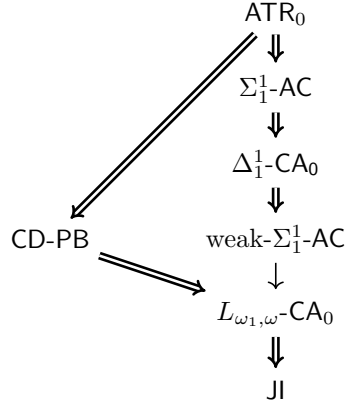
\begin{figure}
  \begin{tikzpicture}[->,shorten >=1pt,auto,node distance=1cm,
                    semithick]
  %\tikzstyle{every state}=[fill=red,draw=none,text=white]

  \node (A)                    {$\ATR$};
  \node (B) [below of=A]       {$\Sigma^1_1\text{-}\mathsf{AC}$};
  \node (C) [below of=B]       {$\Delta^1_1\text{-}\mathsf{CA_0}$};
  \node (D) [below of=C]       {weak-$\Sigma^1_1\text{-}\mathsf{AC}$};
  \node (E) [below of=D]       {$L_{\omega_1,\omega}\text{-}\mathsf{CA}_0$};
  \node (F) [below of=E]       {$\JI$};
  \node (G) at (-3,-3)        {$\DPB$};

  \draw[-Implies, line width=1pt,double distance=1pt] (A)-- (B);
  \draw[-Implies, line width=1pt,double distance=1pt] (B)-- (C);
  \draw[-Implies, line width=1pt,double distance=1pt] (C)-- (D);
  \draw[->] (D)-- (E);
  \draw[-Implies, line width=1pt,double distance=1pt] (E)-- (F);
  \draw[-Implies, line width=1pt,double distance=1pt] (A)-- (G);
  \draw[-Implies, line width=1pt,double distance=1pt] (G)-- (E);

\end{tikzpicture}
\caption{$\DPB$, $\ATR$, and some theories of hyperarithmetic analysis.  The new results are those concerning $\DPB$.  A double arrow indicates a 
strict implication.}\label{fig:1}
\end{figure}
\end{center}

To elaborate on the factors preventing to $\DPB$ from 
being a theory of hyperarithmetic analysis, we prove 
the following generalization of Theorem \ref{thm:01a}
above, establishing that hyperarithmetic generics must
appear in any $\omega$-model of $\DPB$.

\begin{theorem}\label{thm:02}
If $\mathcal M$ is an $\omega$-model of $\DPB$, then for any 
$Z \in M$, there is a $G \in M$ that is $\Delta^1_1(Z)$-generic.
\end{theorem}

As an application, we use 
$\DPB$ to analyze the theorem $\BorelDRT^3_\ell$,
whose statement contains no concept of mathematical logic
apart from that of Borel sets.
(The statement of this theorem can be found in 
Section \ref{sec:bdrt}.)  We show that, under appropriate 
formalization, $\BorelDRT^3_\ell$ is strictly weaker than 
$\ATR$ and shares some properties with the theories of
hyperarithmetic analysis.  It is left open whether 
$\BorelDRT^3_\ell$ is a statement of hyperarithmetic
analysis.

\begin{theorem}
For any finite $\ell \geq 2$, 
the principle $\BorelDRT^3_\ell$ is strictly implied by 
$\ATR$.  Any $\omega$-model of $\BorelDRT^3_\ell$ 
is closed under hyperarithmetic reduction.
\end{theorem}

The first section gives the preliminaries.  In Section 2 
we give the definition of a completely determined Borel code 
and prove its basic properties.  In Section 3 we construct 
an $\omega$-model to separate $\DPB$ from $\ATR$. 
In Section 4 we develop the machinery of \emph{decorating 
trees} which will be used in Sections 5 and 6.  In 
Section 5, we prove that $\DPB$ does not hold in $HYP$. 
In Section 6, we prove Theorem \ref{thm:02}.  This is a 
strictly stronger theorem than the one proved in Section 5, 
but also a bit longer to prove, so Section 5 could be regarded 
as a warm-up.  In Section 7 we give an application to 
the Borel dual Ramsey theorem.
Section 8 contains open questions.

The authors would like to thank Julia Knight and Jindra Zapletal for 
helpful discussions on this topic, and the anonymous referee for 
many suggestions which have made the arguments clearer 
and more accessible.

\section{Preliminaries}
\subsection{Notation, Borel sets and Borel codes}

We typically denote elements of $\omega^{<\omega}$ by $\sigma, \tau$ 
and elements of $2^{<\omega}$ by $p,q$.  We write $\sigma\preceq \tau$ to 
indicate that $\sigma$ is an initial segment of $\tau$, with $\prec$ 
if $\sigma\neq \tau$.    We may also use this notation to indicate when 
a finite string is an initial segment of an infinite string.
For $p \in 2^{<\omega}$, the notation $[p]$ refers to the set 
$\{X \in 2^\omega : p \prec X\}$.
The empty 
string is denoted by $\lambda$.  A string with a single component 
of value $n\in\omega$ is denoted by $\langle n \rangle$.  String 
concatenation is denoted by $\sigma \concat \tau$.  Usually 
we write $\sigma \concat n$ instead of the more technically correct 
but uglier $\sigma \concat \langle n \rangle$.  

If $U$ is a set of strings (for example, a tree, or a coded 
open subset of $2^\omega$), and $\sigma$ is any string, we write 
$\sigma \concat U$ to mean $\{\sigma \concat  \tau : \tau \in U\}$.
If $T$ is a tree and $\sigma \in T$, we write
$T_\sigma$ to mean $\{\tau : \sigma \concat \tau \in T\}$.

The \emph{Borel} subsets of a topological space are the smallest
collection which contains the open sets
and is closed under
complements and countable unions (and thus countable intersections).

A \emph{Borel code} is a well-founded tree 
$T\subseteq \omega^{<\omega}$ whose leaves are labeled 
by basic open sets or their complements, and whose inner nodes
are labeled by $\cup$ or $\cap$.  The Borel set associated to a 
Borel code is defined by induction, 
interpreting the labels in the obvious way.
Any Borel set 
can be represented this way, by applying DeMorgan's laws to push
any complementation out to the leaves.

We use standard recursion-theoretic notation.  The $e$th Turing 
functional is denoted $\Phi_e$.  A pair of natural numbers $(n,m)$ is 
coded as a single natural number $\langle n,m\rangle$ via 
a canonical computable bijection between $\mathbb N$ and $\mathbb N\times \mathbb N$.
Although this notation $\langle n, m\rangle$ could also refer to a string 
with two elements, context will make it clear which type is meant.

\subsection{Reverse mathematics}

We assume the reader is familiar with the program of reverse
mathematics.  The standard reference on this subject is 
\cite{sosa}.  Here we just recall that the principle of 
\emph{arithmetic transfinite recursion} is formulated as follows.
If $X \in 2^\omega$ codes a linear order on some subset of 
$\mathbb N$, let $<_X$ denote
that linear order and (abusing notation) let $X$ also denote the 
domain of the linear order.  Assuming there is a linear order $X$ in the 
context, given $Y \in 2^\omega$ and $a \in X$, we 
let $Y^a$ denote $\{\langle n,b \rangle \in Y : b <_X a\}$.
Given an arithmetic predicate $\phi(n,Z)$, we
define the predicate $H_\phi(X,Y)$ as follows:
$$H_\phi(X,Y) \qquad \equiv \qquad X \text{ is a linear order and } Y = \{\langle n, a\rangle : a \in X \text{ and } \phi(n,Y^a)\}.$$
The principle $\ATR$ is a scheme ranging over arithmetic 
formulas $\phi$, which states that 
for each such $\phi$, if $X$ is a 
well-order, then there is a $Y$ such that $H_\phi (X,Y)$.  
Using $\ACA$, one can show that such $Y$ is unique.  For 
details, see \cite[Section V.2]{sosa}.

In the special case where $\phi(n,Z)$ is the jump operator, that is 
$\phi(n,Z) \equiv n \in Z'$, then any $Y$ satisfying $H_\phi(X,Y)$ 
is called a \emph{jump hierarchy} on $X$.  

The principle of \emph{effective transfinite recursion} is 
defined almost the same as $\ATR$, but using $\Delta^0_1$ formulas 
instead of arithmetical formulas.  In \cite{DFSW} it is shown that 
effective transfinite recursion also goes through in $\ACA$.  

Both $\ATR$ and effective transfinite recursion are used to 
define objects by recursion along a well-order $X$.  If we 
only want to use induction to 
\emph{verify} some arithmetic property 
of a family of objects indexed by $X$, the principle of \emph{arithmetic 
transfinite induction} is used and this principle also holds 
in $\ACA$ (\cite[Lemma V.2.1]{sosa}).

In reverse mathematics, the role of an ordinal is played simply 
by a well-founded linear order.  For certain of our constructions it is 
convenient to have a more structured well-order for which the 
operation of finding a successor is effective.  For that reason we also 
use the terminology of Kleene's $\kO$, which is briefly 
reviewed in the next section.

\subsection{Ordinal notations and pseudo-ordinals}

We assume the reader is familiar with ordinal notations
and pseudo-ordinals.  
A standard reference is \cite{sacksHRT}. 
Here we give just a brief summary of the concepts and 
techniques that we use.  Recall that $\emph{Kleene's O}$, 
denoted $\kO$, is a $\Pi^1_1$-complete subset of $\omega$
consisting of notations for all the computable ordinals, 
where 1 denotes the ordinal 0, $2^a$ denotes the
successor of the ordinal denoted by $a$, and $3\cdot 5^e$
denotes the limit of the ordinals denoted by $\Phi_e(n)$,
provided that $\Phi_e$ is total and for all $n$, $\Phi_e(n) <_\ast \Phi_e(n+1)$ 
(where $<_\ast$ is the transitive closure of
the relation defined by 
$1<_\ast x$ if $x\neq 1$, $x <_\ast 2^x$, 
and $\Phi_e(n) <_\ast 3\cdot 5^e$).  The notation $\leq_\kO$ 
refers to the relation $<_\ast$ restricted to $\kO$.

To avoid 
excessive repetition of the phrase ``denoted by'', henceforth we
will conflate ordinals with their notations.  A given ordinal 
may have many notations, but for each such notation $a$, 
$\{b \in \kO: b<_\kO a\}$ is linearly ordered by $\leq_\kO$, 
so canonical names for the ordinals below $a$ are implied 
by the choice of $a$.  We will also write 
$a+k$ for the $k$th successor of $a$ (instead of the technically 
accurate but more 
cumbersome tower of exponentials), 
and $a-k$ for its $k$th predecessor 
when this exists.  Although $a$ is technically an element of $\omega$, 
it would never make sense to add or subtract an ordinal using the usual 
addition on the natural numbers, so this should not create confusion.
Also, sometimes 
we will take a fixed but unspecified number of successors of 
$a$, and the result is denoted $a + O(1)$.

The set $\{b  \in \kO : b <_\kO a\}$ is c.e. uniformly 
in $a$, because the relation $<_\ast$ is a c.e. relation.
Throughout, we let $p$ denote the computable function such that for each 
$a \in \omega$, we have $W_{p(a)} = \{b \in \omega : b <_\ast a\}$.

The definition of $\kO$ also relativizes to any oracle $X$, producing 
a $\Pi^1_1(X)$-complete set $\kO^X$ with the same properties as above.

The basic tool for working with ordinal notations is effective transfinite 
recursion, which suffices to define a large swath 
of important constructions involving ordinal notations.  
These constructions also relativize (in reverse mathematics 
this corresponds to allowing a real parameter appear in the 
formula $\phi$).  
Here are
two examples which are used in this paper.  (Ranked formulas of $L_{\omega_1,\omega}$ are defined 
in the next subsections). 

\begin{proposition}\label{prop:Hformula}
Given an oracle $X$, an ordinal $a \in \kO^X$, and a number $x \in \omega$, 
there is an $a$-ranked formula of $L_{\omega_1,\omega}$ which holds exactly 
if $x \in H_a^X$, where $H_a^X$ denotes the unique jump hierarchy (relative to $X$) 
on the well-order
$W_{p(a)}^X$.  
\end{proposition}
 
\begin{proof}
The existence of a formula of $L_{\omega_1,\omega}$ 
defining membership in $H_a^X$ follows directly from effective 
transfinite recursion applied to the definition of $H_a^X$.
The fact that the formula can be $a$-ranked uses the normal form theorem 
for simplifying expressions involving bounded quantifiers.  These
simplifications can be carried out effectively.
\end{proof}

\begin{proposition}\label{prop:LObound}
Given an $X$-computable linear order $L$, there is a number $a$
such that $L$ is well-founded if and only if $a \in \kO^X$.  Furthermore,
if $L$ is ill-founded, any descending sequence in $W_{p(a)}^X$ uniformly
computes a descending sequence in $L$.
\end{proposition} 
\begin{proof}
The first statement above is a special case of \cite[Lemma 4.3]{sacksHRT}.  
The second statement is true for the construction in \cite{sacksHRT}, but
not explicitly stated there.  So for the reader's convenience here 
is an alternative construction which establishes both parts of the proposition.  

Let $L'$ denote the linear order with order type $1 + L + 1$.  
Without loss of generality, the least
element of $L'$ is 0 and the greatest element of $L'$ is 1.  
Define a function $e:L' \rightarrow \omega$ as follows.  
Let $e(0) = 1$ (the latter being the code for ordinal 0). 
For $k \in L'$ and $n \in \omega$, let $h(k,n)$ 
denote the $<_{L'}$-greatest element among $\{j\leq n : j<_{L'} k\}$.  Then
for $k\in L'$ with $k\neq 0$, define $e(k)$ by effective transfinite 
recursion as follows.
$$\Phi_{e(k)}^X(n) = \begin{cases} n  \text{ (i.e. the $n$th successor of 0)} & \text{ if }h(k,n) = 0\\
3\cdot 5^{e(h(k,n))} + n & \text{ otherwise.}\end{cases}$$
Then let $a = e(1)$.  It is routine to show that 
$e: (L, <_{L}) \rightarrow (W_{p(a)}^X, <_\ast)$ is order-preserving.  Also, 
the order type of $W_{p(a)}^X$ is $\omega \cdot (1+L)$, with $e$ 
providing a selector for each $\omega$-chain in $W_{p(a)}^X$.  
If $L$ is well-founded, induction along $L$ shows that $a \in \kO^X$.
On the other hand, $e$ and its inverse (the inverse being 
applied to the $\omega$-chains of $W_{p(a)}^X$) 
provide an effective 
correspondence between any descending sequences in $L$ and
in $W_{p(a)}^X$.
\end{proof}

Kleene's $\kO$ also has a $\Sigma^1_1$ superset $\kO^\ast$, defined as the intersection 
of all $X \in HYP$ such that $1 \in X$, $a \in X \implies 2^a \in X$, and 
$$\forall n [\Phi_e(n)\in X \text{ and } \Phi_e(n) <_\ast \Phi_e(n+1)] \implies 3\cdot 5^e \in X.$$
Observe also that $\kO$ is contained in $\kO^\ast$ (the definition of $\kO$ 
is the same, except that to get $\kO$ 
we quantify over all $X$ rather than just those in $HYP$).
Then, since $O^\ast$ is $\Sigma^1_1$, 
it must be a proper extension of $\kO$, and thus
there must be elements in $a \in \kO^\ast \setminus \kO$.
Such elements are called \emph{pseudo-ordinals}.
For all such $a$, $W_{p(a)}$ is an ill-founded linear ordering without
hyperarithmetic descending sequences, and hence must be isomorphic to
$\omega_1^{ck} + \omega_1^{ck}\cdot \mathbb Q + \beta$ for some computable
ordinal $\beta$ (see \cite[Theorem 1.8]{Harrison1968}).
In particular, for every pseudo-ordinal $a$ and every actual ordinal 
$\beta < \omega_1^{ck}$, there exists some $b<_\ast a$ 
which denotes $\beta$.

We will frequently use the following facts about pseudo-ordinals: 
any function on $\kO$ defined by effective transfinite recursion
with $HYP$ parameters
is also defined on all of $\kO^\ast$, and any arithmetic properties 
of the resulting objects also hold for all of $\kO^\ast$,
provided those properties are proved by induction.  These 
facts follow from the more general \cite[Corollary 1.6]{Harrison1968}, 
but they can also be easily seen in our reverse mathematics 
context, using the fact that effective transfinite recursion
and arithmetic transfinite induction hold in $HYP$, but $HYP$ 
believes all pseudo-ordinals are ordinals.  

\subsection{Alternating and ranked trees}\label{sec:breakapart}

The following definition of a ranking for a tree is looser 
than given by some authors.  We only require that the notations 
decrease, rather than the strong requirement that 
$\rho(\sigma) = \sup_n (\rho(\sigma\concat n) + 1)$.
Additionally, it is technically convenient for us to assume
that leaves have the smallest possible rank,
but nothing serious hinges on this.

\begin{definition}
If $T\subseteq \omega^{<\omega}$ is any tree, 
and $\rho : T\rightarrow \kO^\ast$, we say that  $\rho$ 
\emph{ranks} $T$ if 
\begin{enumerate}
\item for all $\sigma$ and $n$ 
such that $\sigma\concat n \in T$, we have 
$\rho(\sigma \concat n) <_\ast \rho(\sigma)$, 
and 
\item for each leaf $\sigma \in T$, $\rho(\sigma) = 1$. 
\end{enumerate}
If $T$ is ranked by $\rho$ and $\rho(\lambda) = a$, 
we say that $T$ is \emph{$a$-ranked} by $\rho$.
\end{definition}
If $T$ is a ranked tree and the name of the ranking 
function is not explicitly given, then its name is $\rho_T$.

Trees appear for us in two contexts: as codes for formulas 
of $L_{\omega_1,\omega}$ and codes for Borel sets.  In both cases,
interior nodes are labeled with one of 
$\{\cap, \cup\}$.  The nicest codes alternate these.

\begin{definition}
If $T\subseteq \omega^{<\omega}$ is a tree with a labeling 
function $\ell$ 
then we say $(T,\ell)$ \emph{alternates} if for every 
$\sigma\concat n \in T$, we have
$\ell(\sigma) \neq \ell(\sigma\concat n)$.
\end{definition}

The main point about alternating trees is that it is 
always safe to assume that we have them. 
If we start with a labeled, $a$-ranked tree,
we can effectively transform it into an alternating
$a$-ranked tree, with no effect on the logic of the tree
(assuming that whatever model we are working in does 
not contain any paths, if the tree is truly ill-founded.)

One small technical detail about this effective transformation will
be used later, so we give the transformation 
explicitly.  The definition is by effective transfinite 
recursion on the rank of the tree.   Given 
an $a$-ranked tree $T$ with a $\cup$ 
at the root, define
$$\operatorname{Alternate}(T) = \{\lambda\} \cup \bigcup_{\sigma \in A(T)}\langle n_\sigma\rangle \concat \operatorname{Alternate}(T_\sigma),$$
where $A(T)$ is the collection of all $\sigma \in T$ such that $\sigma$ is not a 
$\cup$, but each $\tau \prec \sigma$ is a $\cup$, and where $\sigma\mapsto n_\sigma$
is a computable injection from $\omega^{<\omega}$ to $\omega$; define the 
operation analogously when $T$ has a $\cap$ at the root, and 
of course the operation does nothing to a leaf.  We see that 
this operation is just rearranging some subtrees by 
breaking them apart and reattaching them to a higher-ranking parent than 
their original attachment.  The rank 
of any node in $\operatorname{Alternate}(T)$ is inherited from its 
rank in the original tree.  Observe that every level-one subtree of 
the alternated tree (that is, every subtree of the form
$\operatorname{Alternate}(T)_{\langle n_\sigma\rangle}$) is the alternating 
version of a subtree of a single level-one subtree (namely $T_{\langle \sigma(0)\rangle}$)
of the original tree.  In other words, the process of making a tree 
alternate may break apart level-one subtrees, but never mixes them together.

  \subsection{Borel sets in reverse mathematics}

In reverse mathematics, open subsets of $2^\omega$ 
are represented by sets of
strings $p \in 2^{<\omega}$.  If $U$ is such a code, we will 
abuse notation and write $X \in U$ to mean 
that for some $p \in U$, $p \prec X$.  This is 
in addition to also sometimes speaking of the 
strings $p \in U$.  Context will tell which usage is meant.

For arbitrary Borel sets, we will make a more careful 
distinction between code and object.  We restrict attention to 
Borel subsets of $2^\omega$.
A clopen subset $C$ of $2^\omega$ is represented by an element of
$\omega$ which canonically
codes a finite subset $F\subseteq 2^{<\omega}$.  As above,
for $X\in 2^\omega$, we say $X \in C$ if and only if
$p\prec X$ for some $p \in F$.
A code for $C$ as a clopen set gives more
information about $C$ than an open code for the same set,
because the number of elements of
$F$ is computable from the code.  Effectively
in a standard code for a clopen set, one can find a standard
code for its complement.

We take the following as the definition of 
a (labeled) Borel code in reverse mathematics. 

\begin{definition}
A  \emph{labeled Borel code} is a well-founded tree 
$T\subseteq \omega^{<\omega}$, together with a function 
$\ell$ whose domain is $T$, such that if $\sigma$ is 
an interior node, $\ell(\sigma)$ 
is either $\cup$ or $\cap$, and if $\sigma$ is a leaf,
$\ell(\sigma)$ is a standard code for a clopen 
subset of $2^\omega$.
\end{definition}

We call this a \emph{labeled Borel code} instead of a 
\emph{Borel code}, because we have added a labeling 
function to the
original definition to improve 
readability.\footnote{The original definition of a Borel code 
in reverse mathematics \cite{sosa} 
is a well-founded tree $T$ such that for exactly one 
$m \in \omega$, $\langle m \rangle \in T$.

Some conventions are then adopted: if 
$\langle m\rangle \in T$ is a leaf, then $T$
represents a clopen set coded by $m$ according to a 
standard computable look-up; if $\langle m\rangle$ is 
not a leaf, then $T$ represents a union or intersection 
according to the parity of $m$, and the sets to be thus 
combined are those coded by the subtrees 
$T_n = \{\langle n\rangle \concat \sigma : \langle m, n \rangle
\concat \sigma \in T\}$. 
Classically, one can translate easily between this definition and the
definition of Borel code given above, but one direction of the 
translation requires $\ACA$ 
because one cannot effectively determine when 
a node is a leaf.  All the principles considered in this 
paper will imply $\ACA$ over $\RCA$, 
so nothing will be muddled, but for the sake of fastidious 
readers, we will always call these \emph{labeled Borel codes} to 
acknowledge the distinction.}  If $\ell(\sigma) = \cup$ we 
may simply say ``$\sigma$ is a union node'', and similarly for 
$\cap$.  We will also usually suppress mention of $\ell$, 
in an abuse of notation.

If $T$ is a labeled Borel code and $X \in 2^\omega$, 
the existence of an \emph{evaluation map} is used to determine 
whether $X$ is in the set coded by $T$.
\begin{definition}
If $T$ is a labeled Borel code and $X \in 2^\omega$,
an \emph{evaluation map}
for $X \in T$ is a function
$f:T\rightarrow \{0,1\}$ such that 
\begin{itemize}
  \item If $\sigma$ is a leaf, $f(\sigma) = 1$ if and only if $X$
    is in the clopen set coded by $\ell(\sigma)$.
  \item If $\sigma$ is a union node, $f(\sigma) = 1$ if and only
    if $f(\sigma\concat n) = 1$ for some $n\in \omega$.
  \item If $\sigma$ is an intersection node, $f(\sigma) = 1$ if and
    only if $f(\sigma\concat n) = 1$ for all $n \in \omega$.
  \end{itemize}
We say that $X$ is in the set coded by $T$, denoted $X \in \makeset T$, 
if there is an evaluation map $f$ for $X$ in $T$ such that 
$f(\lambda) = 1$.
\end{definition}

Note that $X \in \makeset T$ is a $\Sigma^1_1$ statement.  
Because evaluation maps are naturally constructed by arithmetic
transfinite recursion, $\ATR$ proves that if $T$ is a Borel code and 
$X \in 2^\omega$,
  there is an evaluation map $f$ for $X$ in $T$.  Furthermore, 
  $\ACA$ proves that if an evaluation map exists, then it is unique.
  For detailed proofs, see \cite[Chapter V.3]{sosa}.

Because we are considering these definitions in the context 
of reverse mathematics, there will sometimes be an ill-founded $T$ 
which a model thinks is well-founded.  In these cases, 
the statement $X \in \makeset{T}$ is meaningful inside 
the model, or in the context of a proof inside second
order arithmetic, but is not meaningful outside a model.
However, 
the criteria defining what it means to be an evaluation 
map are absolute, so we can and will construct 
evaluation maps on ill-founded but otherwise coherent 
labeled Borel codes.    
If $T$ is ill-founded, we will never use the notation
$\makeset{T}$ outside of a model.  But if $T$ is 
well-founded, then every $X$ has a unique evaluation map in
$T$.  In that case we give the notation ``$\makeset{T}$''
the obvious meaning of 
$$\{X : \text{the unique evaluation map $f$ for $X$ in $T$ satisfies $f(\lambda) = 1$}\}$$ when 
we refer to it outside the context of a model.

Operations on Borel sets are carried out easily.
Observe that the operation which corresponds to 
complementation on a labeled Borel code is primitive recursive:
just swap all the $\cup$ and $\cap$ labels, and 
replace every clopen leaf label with its complementary label.
\begin{definition}
If $(T,\ell)$ is a labeled Borel code, let $(T,\ell^c)$ denote the 
labeled 
Borel code whose tree is the same, and whose labeling $\ell^c$ is 
complementary to $\ell$ as described above.
\end{definition}
Continuing the abuse of notation, if $T$ is used to refer to some 
$(T,\ell)$, then $T^c$ will be shorthand for $(T,\ell^c)$.
Observe that $\RCA$ proves that if $T$ is a labeled Borel 
code, then $T^c$ is a labeled Borel code.  Similarly, 
if $(T_n)_{n\in\omega}$ is a sequence of labeled Borel codes, 
in $\RCA$ we can construct a code for the intersection 
or union of these sets in the obvious effective way, and $\RCA$ 
will prove that the result is a labeled Borel code.

\subsection{On the maxim that ``Borel sets need $\ATR$''}

Because making meaning out of a standard (labeled)
 Borel code 
requires evaluation maps to be around, $\ATR$ is typically 
taken as the base theory when evaluating theorems involving 
Borel sets.  Even when $\ATR$ is not taken as the base theory, 
theorems involving Borel sets tend to imply $\ATR$.  The 
probable reason for this was observed in \cite{DFSW}.

\begin{theorem}[\cite{DFSW}]\label{thm:1.7}
In $\RCA$, the statement
``For every Borel code $T$, there exists $X$ such that $X \in |T|$ 
or $X \in |T^c|$'' implies $\ATR$.
\end{theorem}
The strength comes from the fact that this statement is asserting the 
existence of an evaluation map for $X$ in $T$.
  If $f$ is an evaluation map for $X$ in $T$,
  then $1-f$ is an evaluation map for $X$ in $T^c$.\footnote{The 
statement in \cite{DFSW} is for original Borel codes, 
but the
proof of the theorem remains valid for labeled Borel codes.}

\begin{restatement}[of Theorem \ref{thm:1.7}]
The statement
    ``For every Borel set, either it or its complement is nonempty''
    is equivalent to $\ATR$ over $\RCA$.
\end{restatement}

This can make the reverse mathematics of some standard theorems 
about Borel sets feel rather empty.  Here is an example. 
Recall that a set
 $A \subseteq 2^\omega$ has the \emph{property of Baire}
  if it differs from an open set by a meager set.  
  That is, there are open sets $U$ and $\{D_n\}_{n\in\omega}$ such that
each $D_n$ is dense,
and for all $X \in \cap_n D_n$, $X \in U \Leftrightarrow X \in A$.
A basic proposition is that
every Borel set has the property of Baire, but what is the 
strength of that proposition in reverse mathematics?  In \cite{DFSW}, the
relevant notions were formalized as follows.

\begin{definition}
A \emph{Baire code} is a collection of open sets
  $U, V, \{D_n\}_{n \in \omega}$ such that $U \cap V = \emptyset$
  and the sets $U\cup V$ and $D_n$ are dense.
\end{definition}

The statement $\mathsf{PB}$ below formalizes the proposition ``Every 
Borel set has the property of Baire.''

\begin{definition}
If $T$ is a Borel code and $U,V,\{D_n\}$ is a Baire code, we say that 
$U,V,\{D_n\}$ is a \emph{Baire approximation} to $T$ if
for all $X \in \cap_n D_n$,
  $X \in U \Rightarrow X \in \makeset T$ and
  $X \in V \Rightarrow X \in \makeset{T^c}$.
\end{definition}

\begin{definition}
Let $\mathsf{PB}$ denote the statement ``Every Borel code has a Baire approximation.''
  \end{definition}
  
\begin{proposition}{\cite{DFSW}}\label{prop:unsatisfactory} In $\RCA$, $\ATR$ is equivalent to $\mathsf{PB}$.
\end{proposition}
\begin{proof} $(\Rightarrow)$ The standard proof uses arithmetic transfinite recursion.\\
$(\Leftarrow)$ If a set has the property of Baire, either it or its complement
is nonempty.
\end{proof}

The reverse direction of this proof is highly unsatisfactory.  The purpose
of this paper is to propose a variant on the definition of a Borel set 
which avoids this and similar unsatisfactory reversals to $\ATR$.

\subsection{Some landmarks between $\ATR$ and $\JI$}

We will end up placing a variant of $\mathsf{PB}$ somewhere in a zoo 
which exists just below $\ATR$.  Much of what is known about this 
region concerns theories, such as $\DCA$, whose $\omega$-models 
are closed under join, hyperarithmetic reduction, and not much more.

\begin{definition}
A \emph{statement of hyperarithmetic analysis} is any statement $S$ 
such that 
\begin{enumerate}
\item whenever $\mathcal M$ is an $\omega$-model which satisfies 
$S$, its second-order part $M$ is closed under hyperarithmetic reduction.
\item For every $Y$, $HYP(Y)$ is the second-order part of an 
$\omega$-model of $S$, where $HYP(Y) = \{X : X\leq_h Y\}$.
\end{enumerate}
A \emph{theory of hyperarithmetic analysis} is any theory which 
satisfies the same requirements as above.
\end{definition}

The original definition of a theory of hyperarithmetic analysis,
due to Steel \cite{Steel1978}, was a theory whose minimum 
$\omega$-model is $HYP$.  The relativized version above was 
first explicitly defined in \cite{Montalban2006}.  At the time of 
Steel's definition, theories such as $\Delta^1_1$-$\mathsf{CA}_0$
were unrelativized (did not permit real parameters from the model). 
However, modern definitions of these theories allow parameters.  
Therefore, the relativized definition of theories of 
hyperarithmetic analysis is the right one for modern usage.

It would be tempting to hope that there would be some theory of 
hyperarithmetic analysis whose $\omega$-models are 
\emph{exactly} the Turing ideals which are closed under
hyperarithmetic reduction, in analogy to the theorems 
characterizing the $\omega$-models of $\RCA$ 
as the Turing ideals, the $\omega$-models
of $\WKL$ as the Scott ideals, and 
the $\omega$-models of $\ACA$ as the Turing ideals 
closed under arithmetic reduction.  However, no 
such theory can exist.

\begin{theorem}\cite{vanwesep77}
For every theory $T$, all of whose $\omega$-models are closed 
under hyperarithmetic reduction, there is a strictly weaker theory $T'$, 
all of whose $\omega$-models are also closed under 
hyperarithmetic reduction,
and which has more $\omega$-models than $T$.
\end{theorem}

Therefore, we are stuck with an infinitely descending zoo 
of statements/theories of hyperarithmetic analysis.

One theory of hyperarithmetic analysis is most relevant to us. 
%\begin{definition}[\cite{m2006}] The principle $\JI$ (jump iteration)
%  is the statement ``For every ordinal $\alpha$
%  and every $X$, if $X^{(\beta)}$ exists for all $\beta < \alpha$,
%  then $X^{(\alpha)}$ exists.''
%\end{definition}
Recall that a formula of 
$L_{\omega_1,\omega}$ is a formula constructed from 
the usual building blocks of first-order logic, together 
with countably infinite conjunctions and disjunctions.
In a language which contains no atomic formulas other 
than {\tt true} and {\tt false}, a formula of $L_{\omega_1,\omega}$
is just a well-founded tree whose interior nodes are 
labeled with either $\cup$ (infinite disjunction) or 
$\cap$ (infinite conjunction), and whose leaves are labeled 
with either {\tt true} or {\tt false}.
An \emph{evaluation map} for a formula of $L_{\omega_1,\omega}$
is defined the same as an evaluation map for an element 
$X$ in a Borel code $T$, except that the evaluation 
map must satisfy $f(\sigma) = 1$ if $\ell(\sigma) =$ {\tt true}
and $f(\sigma) = 0$ if $\ell(\sigma) =${\tt false}.
A formula of $L_{\omega_1,\omega}$ is \emph{completely determined} 
if it has an evaluation map.  Classically, every 
formula of $L_{\omega_1,\omega}$ is completely determined, but 
in weaker theories the witnessing function could fail to exist.
A formula is called \emph{true} if it has a witnessing 
function which maps the formula itself to ${\tt true}$.  

The following definition and result essentially appear in
\cite{Montalban2006}, where $\LwCA$
goes by the name $\mathsf{CDG\text{-}CA}$, and is stated
in terms of games.  The name $\LwCA$ and the definition
given here were introduced in
\cite{MontalbanTHA_Friedman}.

\begin{definition}[similar to \cite{Montalban2006}] The principle $\LwCA$ 
is this statement: If $\{\phi_i : i \in \mathbb N\}$ is a 
sequence of completely determined $L_{\omega_1,\omega}$ formulas, 
then the set $X = \{i : \phi_i \text{ is true}\}$ exists.
\end{definition}

\begin{theorem}[essentially \cite{Montalban2006}] The principle 
$\LwCA$ is a statement of hyperarithmetic analysis.
\end{theorem}

\subsection{Genericity}

The concept of genericity stems directly from category; a 
sufficiently generic member of a set which has the property of Baire
has individual behavior that agrees with the behavior of
a comeager set of reals.  In this subsection we introduce 
the terminology around genericity 
and provide proofs of several folklore results which will
be needed later.

A predicate $P(X)$ is called computable if there is a Turing 
functional $\Delta$ such that for all $X \in 2^\omega$, 
$\Delta(X)$ halts and outputs {\tt true}  or {\tt false} according 
to the truth value of $P(X)$.
A relativized formula of $L_{\omega_1,\omega}$ is a formula 
of $L_{\omega_1,\omega}$ for which the leaves bear 
computable predicates instead of simply {\tt true} or {\tt false}. 
Using the compactness of $2^\omega$ to translate between 
clopen sets and $\{X : \Delta(X) = {\tt true}\}$,
it is immediate that the set of $X$ which satisfy a given 
relativized formulas of $L_{\omega_1,\omega}$
are exactly the members of the Borel set coded by essentially the 
same formula.  If such a formula $\phi$ is computable,
$a$-ranked, and has a union at the root, then 
$\{X \in 2^\omega : \phi(X)\}$ is a $\Sigma^0_a$ 
set and the formula is called a $\Sigma^0_a$ formula.
Of course, the input to the formula could also be a natural 
number, in which case it defines a $\Sigma^0_a$ subset of $\omega$.

If $S \subseteq 2^{\omega}$ is a set of strings, a real $X$ 
is said to \emph{meet} $S$ if for some $p \in S$, 
$p \prec X$, while $X$ is said to \emph{avoid} $S$ if 
some $p \prec X$ has no extension in $S$.  The set $S$ 
is \emph{dense} if every $p \in 2^{<\omega}$ can be 
extended to meet it.
A real $X$ is called \emph{$a$-generic} if $X$ meets or 
avoids every $\Sigma^0_a$ set of strings.  The following propositions, 
which taken together informally assert that set of strings which force a $\Sigma^0_a$ 
statement is $\Sigma^0_a$, are folklore.

\begin{proposition}
Uniformly in a code for a $\Sigma^0_a$
set $A$, there is a $\Sigma^0_a$ code for an open set $U$, 
as well as a uniform sequence 
of $\Sigma^0_a$
codes for dense open sets $D_n$ such that for all $X \in \cap_k D_k$, 
we have $X \in A$ if and only if $X \in U$.  
\end{proposition}
\begin{proof}
This is a straightforward effectivization of the usual
proof that every Borel set has the property of Baire.

The result is is immediate if $a$ is the 0 ordinal 
(in which case all sets described are clopen).  Suppose that it holds
for all $b<_\ast a$.  We have $A = \cup_n A_n$ where each 
$A_n$ is $\Pi^0_{b_n}$ for some $b_n <_\ast a$.  Apply the 
induction hypothesis to the complements $A_n^c$ to 
get a sequence of open sets $U_n$ and a double sequence 
of dense open sets $D_{n,k}$, where each $U_n$ and $D_{n,k}$
have a $\Sigma^0_{b_n}$ code, and any $X \in \cap_k D_{n,k}$ 
is in $A_n^c$ if and only if it is in $U_n$.  Define $V_n$ to be the 
interior of $U_n^c$.  Then each $V_n$ is uniformly 
$\Sigma^0_{b_n+1}$ and thus $\Sigma^0_a$.  We can let 
$U = \cup_n V_n$ and let the sequence of dense open sets 
include all sets $D_{n,k}$, as well as sets of the form $U_n \cup V_n$.
Suppose that $X$ meets all these dense sets. The $X \in A$ 
exactly if $X \in A_n$ for some $n$, which happens exactly if 
$X \not\in U_n$ for some $n$.  Since $X \in U_n \cup V_n$, this 
is true exactly when $X \in V_n$ for some $n$, 
equivalently when $X\in U$.
\end{proof}

\begin{proposition}\label{prop:agen}
If $\phi(X,q)$ is a $\Sigma^0_a$ formula,
there is a $\Sigma^0_a$ formula $R(q)$ such that
for all $a$-generic reals $X$, 
$$\{q : \phi(X,q)\} = \{q : \exists n R(X\upharpoonright n, q)\}$$
\end{proposition}
\begin{proof}
By the previous proposition, uniformly in $q$ there is a code for a
$\Sigma^0_a$ set $U_q \subseteq 2^{<\omega}$
and a sequence of $\Sigma^0_a$ sets $D_k \subseteq 2^{<\omega}$
such that each $D_k$ is dense and for any $X$ that meets 
each $D_k$, 
$X$ meets $U_q$ if and only if $\phi(X,q)$.  Thus 
$R(r, q)$ can be taken to be 
$\exists p(p \in U_q \text{ and } p \prec r)$.
\end{proof}

Now we review some notions from 
higher genericity.  We assume a general familiarity with
 hyperarithmetic theory, and refer the reader to \cite{sacksHRT}
 for definitions and details.  For $G \in 2^\omega$,
 it is well-known that an element $X$ of
 $2^\omega$ is $\Delta^1_1(G)$ if and only if it is $HYP(G)$,
 if and only if there is some $b \in \kO^G$ such that
 $X \leq_T H_b^G$.  

 Recall that if $\Gamma$ is a pointclass, 
$X\in 2^\omega$ is called $\Gamma$-generic
 if $X$ meets or avoids every open set $U$ with a code in $\Gamma$.
 (We have already seen this in the case $\Gamma = \Sigma^0_a$.)
 We are interested in $\Delta^1_1$-generics $G$ with the additional 
property that $\omega_1^{ck} = \omega_1^G$.  By \cite{greenbergmonin}, 
these are precisely the $\Sigma^1_1$-generics.  
However, for our purposes the formal definition of $\Sigma^1_1$-genericity 
seems less useful than the ``$\Delta^1_1$-generic and $\omega_1^{ck}$-preserving'', 
and indeed we never use the equivalence with $\Sigma^1_1$-genericity 
in any way other than as an (accurate) notational shorthand.

The following three propositions must be 
folklore, but we give their proofs here.  Recall that $A$ and $B$ 
are relatively $\Gamma$-generic if $A$ is $\Gamma(B)$-generic 
and $B$ is $\Gamma(A)$-generic.
\begin{proposition}\label{prop:10}
For $G_0,G_1 \in 2^\omega$, we have 
$G_0 \oplus G_1$ is $\Sigma^1_1$-generic if and only if 
$G_0$ and $G_1$ are relatively $\Sigma^1_1$-generic.
\end{proposition}
\begin{proof}
Consider the argument in \cite[Thm. 8.20.1]{ARC} (originally 
due to \cite{Yu2006}), where  
it is shown that $A\oplus B$ is $n$-generic if and only if 
$A$ and $B$ are relatively $n$-generic.  Observe that 
at no point do they make use of the fact that $n$ is finite,
and the same argument goes through if $n$ is replaced with 
any $a \in \kO$.  (Proposition \ref{prop:agen} is used in
the $a \in \kO$ case.)   Therefore the same argument 
shows that $A\oplus B$ is $a$-generic if and only if 
$A$ and $B$ are relatively $a$-generic.
  Observe that 
$A$ is $\Delta^1_1$-generic 
if and only if $A$ is $a$-generic for all $a \in \kO$.

Now suppose that $G_0 \oplus G_1$ is $\Sigma^1_1$-generic.
We will show that $G_0$ is $\Sigma^1_1(G_1)$-generic. 
We have $\omega_1^{G_0\oplus G_1} = \omega_1^{ck} = \omega_1^{G_1}$, 
so it suffices to show that $G_0$ is $\Delta^1_1(G_1)$-generic, 
or equivalently, that $G_0$ is $a$-generic relative to $G_1$ 
for all $a \in \kO$ 
(here we use the fact that $\omega_1^{G_1} = \omega_1^{ck}$).  
This follows from the previous 
paragraph because $G_0\oplus G_1$ is $a$-generic.

On the other hand, if $G_0$ and $G_1$ are relatively $\Sigma^1_1$-generic,
then in particular each is $\Sigma^1_1$-generic,  so 
$\omega_1^{G_0} = \omega_1^{G_1} = \omega_1^{ck}$, and 
by relative $\Sigma^1_1$-genericity, we also have 
$\omega_1^{G_0\oplus G_1} = \omega_1^{ck}$.\footnote{Although
\cite{greenbergmonin} does relativize, the conclusions here can be established
 without using that relativization. It 
 suffices to show that at least one of $G_0$, $G_1$ is $\omega_1^{ck}$-preserving.
 Suppose that $G_0$ computes a linear order 
 of order type $\omega_1^{ck}$.  Then $\{X : X \text{ computes a linear order of order type $\omega_1^{ck}$}\}$ is $\Delta^1_1(G_0)$ and meager, so $G_1$ 
 does not belong to it.  Thus $\omega_1^{G_1} = \omega_1^{ck}$.}
  Therefore it suffices 
to show that $G_0 \oplus G_1$ is $\Delta^1_1$-generic, or 
equivalently, that it is $a$-generic for all $a \in \kO$.  
This follows because $G_0$ and $G_1$ are relatively $a$-generic 
for all $a \in \kO$.
\end{proof}

The following two propositions will be used later in a relativized form.  To 
reduce clutter, we do not write the relativized form, but the reader 
can verify that all the arguments relativize.

\begin{proposition}\label{prop:9}
If $G_0\oplus G_1$ is $\Sigma^1_1$-generic, then 
$\Delta^1_1(G_0) \cap \Delta^1_1(G_1) = \Delta^1_1$.
\end{proposition}
\begin{proof}
If $X \in \Delta^1_1(G_0) \cap \Delta^1_1(G_1)$, then 
since $\omega_1^{G_0} = \omega_1^{G_1} = \omega_1^{ck}$, 
there are $a \in \kO$ and indices $e$ and $f$ such 
that $X = \Phi_e(H_a^{G_0}) = \Phi_f(H_a^{G_1})$.  Consider 
the set $$W =\{Y \oplus Z : \Phi_e(H_a^Y) = \Phi_f(H_a^Z)\}.$$  This 
set is $\Delta^1_1$, so it has the property of Baire, 
and in particular there is a $\Delta^1_1$ open set $V$ 
such that every sufficiently generic $Y\oplus Z$ 
is an element of $V$ 
if and only if it is an element of $W$.  Here the amount of genericity 
needed is not full $\Delta^1_1$-genericity, but rather 
$c$-genericity, where $c = a + O(1)$.  To see that $c$-genericity 
suffices, first use Proposition \ref{prop:Hformula} to write the 
defining property of $W$ as a $c$-ranked relativized 
formula of $L_{\omega_1,\omega}$, then apply Proposition \ref{prop:agen}.

Since $G_0\oplus G_1$ 
is $\Delta^1_1$-generic and in $W$, it is in $V$.  
Let $p, q \in 2^{<\omega}$ 
be such that $p\prec G_0, q\prec G_1$ and $p\oplus q \in V$.
Now let $Y$ be any $c$-generic, hyperarithmetic real with $p \prec Y$.  
Then 
since $G_1$ is $\Delta^1_1$-generic, it is $\Delta^1_1$-generic 
relative to $Y$, so in particular it is $c$-generic relative to $Y$. 
Using the ordinal version of \cite[Thm. 8.20.1]{ARC} a second time, we conclude that
$Y \oplus G_1$ is $c$-generic, and meets $V$.  Therefore, 
$Y \oplus G_1 \in W$, and we obtain a $\Delta^1_1$ formula for $X$,
that is, $X = \Phi_e(H_a^Y)$.
  \end{proof}

\begin{proposition}\label{prop:11}
Let $G_0$ be $\Sigma^1_1$-generic and $P$ a hyperarithmetic predicate.
If there is a $Y \in \Delta^1_1(G_0)$ 
such that $P(Y)$ holds,
then for all $\Delta^1_1$-generic $G_1$, there is a 
$Y \in \Delta^1_1(G_1)$ such that $P(Y)$ holds.
\end{proposition}
\begin{proof}
Since $\omega_1^{ck} = \omega_1^{G_0}$, there is some $a \in \kO$ 
and an index $e$ such that $Y = \Phi_e(H_a^{G_0})$.  Then 
$R(X) := \exists e P(\Phi_e(H_a^{X}))$ is a hyperarithmetic predicate 
that holds of $G_0$ and holds of $p\concat G_0$ for 
any $p \in 2^{<\omega}$.  Therefore, for any $\Delta^1_1$-generic 
$G_1$, $R(G_1)$ holds. 
\end{proof}

Finally, we remark that for any $Z$, the set of $\Delta^1_1(Z)$-generics 
is $\Sigma^1_1(Z)$.  This is because
$$X \text{ is $\Delta^1_1(Z)$-generic } \iff \forall Y \in \Delta^1_1(Z) [ X \text{ is 1-generic relative to $Y$}].$$

 \section{Completely determined Borel codes}

We propose the following variation on the definition of a Borel code.  We
shall see that when this variant is used, the unsatisfactory 
shortcut in
Proposition \ref{prop:unsatisfactory} vanishes, 
and indeed the reversal no longer holds. 

\begin{definition}
A labeled Borel code $T$ is called \emph{completely determined} if
  every $X\in 2^\omega$ has an evaluation map in $T$.  A \emph{completely determined
    Borel code} is a labeled Borel code that is completely determined.
\end{definition}

When we formalize statements in reverse mathematics, in order to not 
conflict with existing convention, we will 
say \emph{completely determined Borel set} to indicate when the formalized 
version of the statement should call for a completely determined Borel code. 

The following facts are immediate.
\begin{proposition} In $\RCA$,
\begin{enumerate}
\item If $T$ is a completely determined Borel code, 
then $T^c$ is also a completely determined Borel code.
\item For every completely determined Borel set $A$ and $X \in 2^\omega$,
either $X \in A$ or $X \not\in A$.
\end{enumerate}
\end{proposition}

With only a slight amount of effort, we also have the following.
\begin{proposition} In $\RCA$,
if $A$ is a completely determined Borel set and $h:2^\omega\rightarrow 2^\omega$ 
is continuous, then $h^{-1}(A)$ is a completely determined Borel set.
\end{proposition}
\begin{proof}
Let $T$ be a completely determined Borel code and $h:2^\omega\rightarrow 2^\omega$ 
a continuous function. Then $h$ is encoded by a sequence of pairs 
$(p_1,q_1), (p_2,q_2)\dots$ from $2^{<\omega}\times 2^{<\omega}$, which are compatible in the sense that 
$p\preceq p' \implies q\preceq q'$ whenever $(p,q),(p',q')$ are in $h$.  If $(p,q)$ 
is in $h$, 
it means that $p \prec X$ implies that $q \prec h(X)$.  For $h$ to be well-defined, 
we must have that for each $X$, there are arbitrarily long $q$ for which $q \prec h(X)$. 
Define $S$ by starting with $S=T$ and 
modifying each leaf $\sigma \in T$ as follows:
\begin{enumerate}
\item In $S$, $\sigma$ is a union.
\item For each $n$, $\sigma \concat n \in S$ and is a leaf.
\item If $U$ is the clopen set attached to $\sigma$ in $T$, 
let $\sigma\concat n$ be labeled with a code for the clopen 
subset of $h^{-1}(U)$ defined by
$$\cup \{[p_i] : (p_i,q_i) \in h, i< n, [q_i] \subseteq U\}$$
\end{enumerate}
We claim that $S$ is completely determined and $X \in \makeset S$ if and only 
if $h(X) \in \makeset{T}$.  Let $f$ be an evaluation map for 
$h(X)$ in $T$.  We claim that $f$ can be extended to 
an evaluation map for $X$ in $S$ by adding $f(\sigma\concat n) = 1$ 
if and only if $X$ is in the clopen set attached to $\sigma\concat n$ 
in $S$.  One only needs to check that the logic of the evaluation 
map is correct at each $\sigma$ which was a leaf in $T$.  
\end{proof}

The fact that Borel sets are closed under countable union, which 
was trivial using the standard definition of a Borel set,
has quite some power for completely determined Borel sets.

\begin{proposition}\label{prop:countable_union}
In $\RCA$, the statement ``A countable union of completely determined Borel 
sets is a completely determined Borel set'' is equivalent to $\LwCA$.
\end{proposition}
\begin{proof}
If $\{T^k : k \in \mathbb N\}$ are completely determined Borel codes,
and
$T = \{\lambda\} \cup \{\langle k \rangle \concat \sigma : 
\sigma \in T^k\}$, we claim that, assuming $\LwCA$, 
$T$ is completely determined.
Fixing $X$, let $\phi_{k,\sigma}$ be the formula obtained by 
replacing each clopen set at each leaf of $T^k_\sigma$ 
by {\tt true} or {\tt false} according to whether $X$ 
is in each clopen set.  Any evaluation map for $X$ in 
$T^k$ can be restricted to an evaluation map for $X$ in 
$T^k_\sigma$, which is an evaluation map for $\phi_{k,\sigma}$, 
so all these formulas are completely determined.  One obtains an 
evaluation map for $X$ in $T$ by letting 
$f(\sigma) = 1$ if and only if $\phi_{k,\sigma}$ is true, 
and then non-uniformly filling in $f(\lambda)$ to its 
unique correct value.

Conversely, if $\{\phi_k : k \in \mathbb N\}$ 
are completely determined, these formulas can be modified at the 
leaves to become completely determined Borel codes $T^k$ 
for $\emptyset$ or 
$2^\omega$ according to whether they are true or false.
Defining $T$ as above, any evaluation map $f$ for $T$ 
satisfies $f(\langle k\rangle) = 1$ if and only if $\phi_k$ is true.
\end{proof}

Now we consider the completely determined variant of $\mathsf{PB}$.

\begin{definition} Let $\DPB$ be the statement ``Every completely determined
    Borel set has the property of Baire.''
    \end{definition}

Our main question is: what is the reverse mathematics strength of $\DPB$?

\begin{proposition}\label{prop:dpb_implies_lwca}
 In
$\RCA$, $\DPB$ implies $\LwCA$.
\end{proposition}
\begin{proof}
Any sequence $\{\phi_k:k\in\mathbb N\}$
 of completely determined formulas of $L_{\omega_1,\omega}$ 
can be modified at the leaves
to produce a sequence of completely determined Borel codes 
which code either $[0^k1]$ or $\emptyset$ depending 
on whether $\phi_k$ is true or false.  The union of
these remains completely determined because each $X$ 
passes through at most one of these sets.  Any Baire 
approximation to $\cup_{k : \phi_k \text{ is true }} [0^k1]$
computes $\{k : \phi_k \text{ is true}\}$.
\end{proof}

This places $\DPB$ somewhere in the general area of 
$\ATR$ and the theories of hyperarithmetic analysis. 
If $\DPB$ were equivalent to $\LwCA$, our variant would 
be subject to the same kinds of critique that we made of the 
original definition (all the strength of the theorem 
coming essentially from Proposition \ref{prop:countable_union}).
However, it turns out $\DPB$ is equivalent to none of the 
principles mentioned so far.

When considering how to show that $\DPB$ is strictly
weaker than $\ATR$, it is informative to consider the
usual proof that every Borel set has the property of
Baire.  This proof uses arithmetic transfinite recursion
on the Borel code of the given set.   It constructs not only
a Baire code for the given set, but also Baire codes for all Borel
sets used to build up the given one.  Below, we give the name
\emph{Baire decomposition} to this extended object that
$\ATR$ would have created.  Superficially, $\DPB$
would seem weaker than the statement ``every completely determined Borel set
has a Baire decomposition'', and one might wonder whether
the additional information in the Baire decomposition
carries any extra strength.  The
purpose of the rest of this section is to show
that it does not (Proposition
\ref{prop:noextra}), and to mention exactly how a Baire
approximation is constructively obtained from a Baire
decomposition (Proposition \ref{prop:decomposition_implies_approximation}).
The point is that any model separating $\DPB$ from $\ATR$ will
need another method of producing an entire Baire decomposition,
not just the Baire approximation.

\begin{definition}
Let $T$ be a completely determined Borel code. A \emph{Baire decompositon} for $T$ is a collection of open sets $U_{\sigma}$ and $V_{\sigma}$ for $\sigma \in T$ such that 
for each $\sigma \in T$ and each $p \in 2^{< \omega}$, 
\begin{enumerate}
\item $U_{\sigma} \cup V_{\sigma}$ is dense and $U_{\sigma} \cap V_{\sigma} = \emptyset$, 
\item if $\sigma$ is a leaf, then $U_{\sigma}$ is dense in the clopen set $C$ coded by $\ell(\sigma)$ and $V_{\sigma}$ is dense in $C^c$, 
\item if $\sigma$ is a union node, then $U_{\sigma}$ is dense in $\bigcup_n U_{\sigma^{\smallfrown}n}$ and $\bigcup_n U_{\sigma^{\smallfrown}n}$ is dense in $U_\sigma$,
\item if $\sigma$ is an intersection node, then $V_{\sigma}$ is dense in $\bigcup_n V_{\sigma^{\smallfrown}n}$ and $\bigcup_n V_{\sigma^{\smallfrown}n}$ is dense in $V_\sigma$. 
\end{enumerate}
\end{definition}

\begin{proposition}[$\ACA$]\label{prop:noextra}
$\DPB$ implies that every completely determined Borel set has a Baire decomposition.
\end{proposition}
\begin{proof}
  Let $T$ be a completely determined Borel code.
  Informally, we partition the space into countably many disjoint clopen
  pieces (plus one limit point) and put an isomorphic copy of the
  set coded by $T_\sigma$ in the $\sigma$th piece.  Then we show that
  a Baire approximation to this disintegrated set can be translated
  back to a Baire decomposition for the original set coded by $T$.

More formally, for any $p \in 2^{<\omega}$, let
$T[p]$ denote the labeled Borel code for $\{p\concat X : X \in |T|\}$.
This is an effective operation on codes.  Recall that
each leaf codes a clopen set by a finite list $F \subseteq 2^{<\omega}$.
By replacing each such $F$ with $\{p\concat q : q \in F\}$,
we achieve the desired effect.

For any $\sigma \in \omega^{<\omega}$, 
let $\ceil \sigma$ be a natural number which codes $\sigma$ 
in a canonical way.  Define $S$ to be the labeled Borel code 
$$S = \{\lambda\} \cup \{\ceil \sigma \concat \tau : \tau \in T_\sigma[0^{\ceil \sigma}1], \sigma \in T\}$$
where $\lambda$ is a $\cup$ and all other labels are inherited from 
the $T_\sigma[0^{\ceil \sigma}1]$.
Then $S$ is completely determined: for any $X$, if $X = 0^\omega$, then the 
identically zero map is an evaluation map for $X$; if $X = 0^n1\concat Y$, 
then if $f$ is an evaluation map for $Y$ in $T$ and $n = \ceil \sigma$, 
an evaluation map $g$ for $X$ in $S$ can be defined by letting 
$g(\ceil \sigma \concat \tau) = f(\sigma\concat\tau)$ on 
$$\{\ceil \sigma \concat \tau : \tau \in T_\sigma[0^{\ceil \sigma}1]\},$$
$g(\lambda) = f(\sigma)$, and $g$ identically zero elsewhere. 
Therefore, for all $Y$ and $\sigma$,
$$0^{\ceil \sigma} 1 \concat Y \in \makeset S \iff Y \in \makeset{T_\sigma}.$$

Now suppose that $(U,V,\{D_k\}_{k\in\omega})$ is a Baire approximation 
for $S$. Then define $U_\sigma = \{q : 0^{\ceil \sigma}1\concat q \in U\}$ 
and $V_\sigma = \{q : 0^{\ceil \sigma}1\concat q \in V\}$.  
We claim that $(U_\sigma,V_\sigma)_{\sigma \in T}$ is a Baire
decomposition for $T$.
Property
(1) of a Baire decomposition is clear.  
For property (2), this follows because if $[q]$ is contained in the 
clopen set $|T_\sigma|$, suppose for contradiction that there is 
$r$ extending $q$ with $[0^{\ceil \sigma}1\concat r] \subseteq V$.  Then for all $X\in [r]$,
we have $X \in |T_\sigma|$ and thus $0^{\ceil \sigma}1\concat X \in U$, a contradiction.  Therefore 
$V_\sigma \cap [q] = \emptyset$, so $U_\sigma$ is dense in $[q]$.  A similar 
argument applies to establish that $V_\sigma$ is dense in $[q]$ 
if $[q]$ is contained $|T_\sigma^c|$.
For property (3), 
letting $\sigma$ be a union node and $p \in 2^{<\omega}$, we will show 
that $U_\sigma$ is dense in $[p]$ if and only if $\cup_n U_{\sigma\concat n}$ 
is dense in $[p]$.
Suppose that $\cup_n U_{\sigma\concat n}$ is not dense in $[p]$.  
Let $q$ extend $p$ such that for all $n$, 
$U_{\sigma\concat n} \cap [q] = \emptyset$.  Then define $Y$ 
so that $q \prec Y$ and the following collection of comeager events occur:
\begin{enumerate}[(i)]
\item For all $n$, $Y \in V_{\sigma\concat n}$
\item\label{i2} For all $n$, $0^{\ceil {\sigma\concat n}}1\concat Y \in \cap_k D_k$
\item $Y \in U_\sigma \cup V_\sigma$
\item\label{i4} $0^{\ceil \sigma}1\concat Y \in \cap_k D_k$
\end{enumerate}
The first comeager event guarantees that $Y \in V_{\sigma\concat n}$ for 
all $n$.  Together with second comeager event this implies that 
$0^{\ceil{\sigma\concat n}}1\concat Y \not\in \makeset{S}$, and 
therefore $Y \not\in \makeset{T_{\sigma\concat n}}$. Therefore,
$Y \not\in \makeset{T_\sigma}$.  In the third dense event, 
if we had $Y \in U_\sigma$, the fourth comeager event would 
imply that $Y \in \makeset{T_\sigma}$; therefore it must be 
that $Y \in V_\sigma$, and so $U_\sigma$ is not dense in $[p]$. 
On the other hand, if $U_\sigma$ is not dense in $[p]$, 
then assuming $\cup_n U_{\sigma\concat n}$ is dense in $[p]$ 
leads to a contradiction, for we may similarly define $Y$ 
to meet $V_\sigma\cap [p]$ and $\cup_n U_{\sigma\concat n}$, 
while also satisfying (\ref{i2}) and (\ref{i4}).

The proof of (4) is similar to the proof of (3).
\end{proof}

Turning a Baire decomposition into a Baire approximation involves 
extracting the comeager set on which the approximation should 
hold.  The following proposition gives a canonical sequence of 
dense open sets which suffices for this.

\begin{proposition}[$\mathsf{ACA}_0$]\label{prop:decomposition_implies_approximation}
Let $T$ be a completely determined Borel code and $(U_{\sigma},V_{\sigma})_{\sigma \in T}$ 
be a Baire decomposition for $T$. Let $\{D_n\}_{n \in \omega}$ 
consist of the following dense open sets:
\begin{enumerate}
\item $U_{\sigma} \cup V_{\sigma}$ for $\sigma \in T$, 
\item $V_{\sigma} \cup \bigcup_n U_{\sigma^{\smallfrown}n}$ for 
$\sigma \in T$ a union node, and 
\item $U_{\sigma} \cup \bigcup_n V_{\sigma^{\smallfrown}n}$ for $\sigma \in T$ an intersection node.
\end{enumerate}
 Then, $(U_{\lambda}, V_{\lambda}, \{D_n\}_{n\in\omega})$
is a Baire approximation for $T$. 
\end{proposition}
\begin{proof}
The properties of a Baire decomposition suffice to ensure that 
$(U_\lambda, V_\lambda,  \{D_n\}_{n\in\omega})$ is a Baire code.  
We must show that if $X \in \cap_n D_n$, then $X \in U_\lambda \implies 
X\in \makeset{T}$ and $X \in V_\lambda \implies X \in \makeset T^c$.
Fix $X\in \cap_n D_n$.  We prove by arithmetic transfinite induction 
that for all $\sigma \in T$, if $X \in U_\sigma$ then 
$X \in \makeset{T_\sigma}$ and if $X \in V_\sigma$ then 
$X \in \makeset{T_\sigma^c}$.  This holds when $\sigma$ is a leaf.

If $\sigma$ is a union node, suppose $X \in U_\sigma$.  
Then $X \not\in V_\sigma$, but 
$X \in V_\sigma \cup \bigcup_n U_{\sigma\concat n}$, so 
$X \in U_{\sigma\concat n}$ for some $n$.  Then the induction 
hypothesis gives us $X \in |T_{\sigma \concat n}|$, so $X \in |T_\sigma|$ 
since $\sigma$ is a union node.

On the other hand, if $X \in V_\sigma$, let $p \prec X$ be such that $p \in V_\sigma$.
Then $U_\sigma \cap [p] = \emptyset$, so 
$\cup_n U_{\sigma\concat n} \cap [p] =\emptyset$.  Recall that $U_{\sigma\concat n} \cup V_{\sigma\concat n}$ is dense by definition.  So for each $n$, 
$V_{\sigma \concat n}$ is dense in $[p]$.  Therefore, $X$ meets 
each $V_{\sigma\concat n}$, so by induction 
$X \in \makeset{T_{\sigma\concat n}^c}$ holds for all $n$.  
Therefore, $X \in \makeset{T_\sigma^c}$.

The case where $\sigma$ is an intersection node is similar.
\end{proof}

 \section{$\DPB$ does not imply $\ATR$}

 Our non-$\ATR$ method of producing a Baire decomposition involves
 polling sufficiently generic $X$ to see whether they are in or
 out of a given set.  For our purposes, sufficiently generic means 
 $\Sigma^1_1$-generic.

Let $G = \bigoplus_i G_i$ be a $\Sigma^1_1$ generic.  Let
$\mathcal M = \bigcup_n \Delta^1_1(\bigoplus_{i<n} G_i)$.
This is the model which will be used to separate $\DPB$ 
and $\ATR$.  But first, some lemmas.

\begin{lemma}\label{lem:LwCA}
$\mathcal M \models \LwCA$.  Furthermore, whenever 
$F \subseteq \omega$ is finite and 
%$j \not\in F$ and 
the completely determined sequence 
of formulas $\{\phi_k : k \in \mathbb N\}$ is in 
$\Delta^1_1(\bigoplus_{i \in F} G_i)$, we also have 
%that $\{\phi_i : k \in \mathbb N\}$ is completely determined 
%in $\Delta^1_1(\bigoplus_{i \in F\cup\{j\}} G_i)$ and
$$\{ k : \phi_k \text{ is true in }M\} \in 
\Delta^1_1\left(\bigoplus_{i \in F} G_i\right).$$ 
\end{lemma}
\begin{proof}
We begin with three facts. First, applying Proposition \ref{prop:10} to
the decomposition $G = \bigoplus_{i \in F} G_i \oplus \bigoplus_{i
  \not \in F} G_i$, we conclude that $\bigoplus_{i \not \in F} G_i$ is
$\Sigma^1_1(\bigoplus_{i \in F} G_i)$-generic.

Second, fix $j \not \in F$. Applying Proposition \ref{prop:10} to $G =
G_j \oplus \bigoplus_{i \neq j} G_i$, we have that $G_j$ is
$\Sigma^1_1(\bigoplus_{i \neq j} G_i)$-generic and hence $G_j$ is
$\Sigma^1_1(\bigoplus_{i \in F} G_i)$-generic.

Third, fix $j_0, j_1 \not \in F$ with $j_0 \neq j_1$. By the same
argument, we have that $G_{j_0}$ is $\Sigma^1_1(G_{j_1} \oplus
\bigoplus_{i \in F} G_i)$-generic and that $G_{j_1}$ is
$\Sigma^1_1(G_{j_0} \oplus \bigoplus_{i \in F} G_i)$-generic. By
Proposition \ref{prop:9} relativized to $\bigoplus_{i \in F} G_i$, it
follows that $\Delta^1_1(G_{j_0} \oplus \bigoplus_{i \in F} G_i) \cap
\Delta^1_1(G_{j_1} \oplus \bigoplus_{i \in F} G_i) =
\Delta^1_1(\bigoplus_{i \in F} G_i)$.

We now apply Proposition \ref{prop:11} relativized to $\bigoplus_{i
  \in F} G_i$. Fix $j \not \in F$ and $k \in \omega$. Since
$\bigoplus_{i \not \in F} G_i$ is $\Sigma^1_1(\bigoplus_{i \in F}
G_i)$-generic, $G_j$ is $\Sigma^1_1(\bigoplus_{i \in F} G_i)$-generic
and there is a $\Delta^1_1(G)$ evaluation map for $\phi_k$, it follows
that $\phi_k$ is completely determined in $\Delta^1_1(G_j \oplus \bigoplus_{i \in
  F} G_i)$. Because this holds for any $j \not \in F$, $\phi_k$ is
completely determined in $\Delta^1_1(\bigoplus_{i \in F} G_i)$ by the third fact
above. Since $\mathsf{L}_{\omega_1, \omega}$-$\mathsf{CA}$ is a theory
of hyperarithmetic analysis, the conclusion follows.
\end{proof}

\begin{proposition}
$\mathcal M \not\models \ATR$.
\end{proposition}
\begin{proof}
Let $a^\ast \in \kO^\ast$.  Then $\mathcal M$ believes that 
$a^\ast$ is an ordinal.  For if there were a 
$\Delta^1_1(G)$-computable descending sequence in $a^\ast$,
then for some
$b \in \kO$ (here we use the fact that $\omega_1^{ck} = \omega_1^G$)
the statement $R(X) :$ ``$H_b^X$ computes a descending 
sequence in $a^\ast$'' is a hyperarithmetic predicate
which holds of $G$.  As $R$ 
holds of $p\concat G$ for any $p \in 2^\omega$,
the set of $X$ for which $R$ holds is comeager
(since each $p\concat G$ is $\Sigma^1_1$-generic, there 
can be no $p$ which forces $\neg R(X)$, therefore 
the set of $p$ which force $R(X)$ is dense).
Furthermore, $R(X)$ is $\Sigma^0_{b + O(1)}$ , so $R(X)$ 
holds for any $X$ which is $b+O(1)$-generic.  There is a 
hyperarithmetic such $X$.  But then $H_b^X$ is also 
hyperarithmetic, contradicting that $a^\ast$ has no 
hyperarithmetic descending sequence.  So $a^\ast$ is well-founded, 
according to $\mathcal M$.

For contradiction, suppose there were a jump hierarchy on 
$a^\ast$ in $\Delta^1_1(G)$.  Then for some $b \in \kO$, 
$R(X) :=$ ``$H_b^X$ computes a jump hierarchy on $a^\ast$'' 
is again a $\Sigma^0_{b+O(1)}$ predicate,
where $R$ holds of $G$.  (Recall that
being a jump hierarchy on $a^\ast$ is just a $\Pi^0_2$ property).
Arguing as above, hyperarithmetically in 
any $b+O(1)$-generic $X$, we would have a jump 
hierarchy on $a^\ast$, which is impossible since $a^\ast$ has 
no hyperarithmetic jump hierarchy.
\end{proof}

Below, the way that $\mathcal M$ can produce a Baire decomposition 
without resorting to arithmetic transfinite recursion is by 
polling a sufficiently generic element $G_i$ about whether 
$p\concat G_i \in \makeset{T}$ while varying $p \in 2^{<\omega}$ 
to get a complete picture of the comeager behavior of $T$.

 \begin{theorem}\label{thm:1}There is an $\omega$-model of $\DPB$ that does
  not satisfy $\ATR$.
  \end{theorem}
\begin{proof}
Let $\mathcal M$ be as above.
Let $T \in M$ be a labeled Borel code which is completely determined 
in $M$.  We consider the case where $T \in \Delta^1_1$;
the case where $T \in \Delta^1_1(\bigoplus_{i<n} G_i)$ 
follows by relativization. Since $T$ is completely determined, for each $G_i$ and 
each $p \in 2^{<\omega}$, the statements
$p\concat G_i \in \makeset{T_\sigma}$
can be understood as a completely determined formulas of $L_{\omega_1,\omega}$
(by replacing the leaves of $T_\sigma$ with 0 or 1 
according to whether $p\concat G_i$ is in those 
sets). These formulas are uniformly $\Delta^1_1(G_i)$.  Therefore,
by Lemma \ref{lem:LwCA}, we have
$$\{(\sigma,p): p\concat G_i \in \makeset{T_\sigma}\} 
\in \Delta^1_1(G_i)$$
Therefore, for each $i$, $\Delta^1_1(G_i)$ contains the sequence $(U_\sigma^i,V_\sigma^i)_{\sigma\in T}$ defined by
$$U_\sigma^i = \{ p : \forall q \succeq p, q\concat G_i \in \makeset{T_\sigma}\},
\qquad V_\sigma^i = \{p : \forall q \succeq p, q \not\in U_\sigma^i\}.$$

We claim that for each $i\neq j$ and for each
$\sigma \in T$, the collections
$$(U_{\sigma\concat\tau}^i,V_{\sigma\concat\tau}^i)_{\tau\in T_\sigma}, \qquad 
(U_{\sigma\concat\tau}^j,V_{\sigma\concat\tau}^j)_{\tau\in T_\sigma}$$
are Baire decompositions for $T_\sigma$, and are equal.  The proof 
(for fixed $i,j$) is
carried out inside of $\mathcal M$ 
by arithmetic transfinite induction on the rank of $\sigma$ in $T$.
Specifically, we claim that

\begin{enumerate}
\item[(1)] If $\sigma$ is a leaf, then $U^i_{\sigma} =$ the clopen set coded by $\sigma$ and $V_{\sigma} = ({U}^i_{\sigma})^c$.
\item[(2)] If $\sigma$ is a union node, then for all $p \in 2^{< \omega}$, $p \in U^i_{\sigma}$ if and only if $\bigcup_n U^i_{\sigma^{\smallfrown}n}$ is dense in $[p]$.
\item[(3)] If $\sigma$ is an intersection node, then for all $p \in 2^{< \omega}$, $p \in V^i_{\sigma}$ if and only if $\bigcup_n V^i_{\sigma^{\smallfrown}n}$ is dense in $[p]$.
\item[(4)] $U^j_{\sigma} = U^i_{\sigma}$ (and thus $V_\sigma^j = V_\sigma^i$).
\end{enumerate}

Note that the definition of the $V_\sigma^i$ in terms of $U_\sigma^i$ guarantees 
that $U_\sigma^i \cup V_\sigma^i$ is dense and 
$U_\sigma^i \cap V_\sigma^i = \emptyset$, and the remaining parts of the 
claim suffice to establish that we have a Baire decomposition.

When $\sigma$ is a leaf, it is clear that $U_\sigma^i$ and $U_\sigma^j$ 
consist of precisely those $p$ such that $[p]$ is contained in 
the clopen set coded by $\ell(\sigma)$.

Now fix an interior node $\sigma$. By induction, we can assume that for all $\tau \in T$ properly extending $\sigma$, condition (4) holds, so we drop 
the superscripts and denote these open sets by $U_{\tau}$ and $V_\tau$. Since 
Properties (1)-(3) hold for $\rho$ extending such $\tau$, we have that 
$(U_{\rho},V_{\rho})_{\rho\in T_\tau}$ are a Baire decomposition for $T_{\tau}$. 
We let 
$D_{m,\tau}$ denote the canonical sequence of dense open sets 
from Proposition \ref{prop:decomposition_implies_approximation}
corresponding to this Baire decomposition. 
Since $(D_{m,\tau})_m \in \Delta^1_1(G_i) 
\cap \Delta^1_1(G_j)$, so by 
Proposition \ref{prop:9},
$(D_{m,\tau})_m \in \Delta^1_1$.  Therefore, for all 
$p \in 2^{<\omega}$, we have $p\concat G_i, p\concat G_j \in \cap_m D_{m,\tau}$.
Therefore, if $p\concat G_i \in U_\tau$, then 
$p\concat G_i \in \makeset{T_\tau}$, and if 
$p\concat G_i \in V_\tau$, then 
$p\concat G_i \not\in \makeset{T_\tau}$, and the same holds for $G_j$.

Suppose that $\sigma$ is a union node.
To prove ($\Rightarrow$) in (2), fix $q \in U^i_{\sigma}$. 
We need to show that 
$\{ r \in 2^{< \omega} : q^{\smallfrown}r \in 
\bigcup_n U_{\sigma^{\smallfrown}n} \}$ is dense. 
For a contradiction, suppose $[q^{\smallfrown}r_0] \cap \bigcup_n U_{\sigma^{\smallfrown}n} = \emptyset$ for some fixed $r_0$. To obtain a contradiction, we will show 
that for all $n$, we have 
$q^{\smallfrown}r_0^{\smallfrown}G_i \not \in |T_{\sigma^{\smallfrown}n}|$.
 Since $\sigma$ is a union node, it follows that 
$q^{\smallfrown}r_0^{\smallfrown}G_i \not \in \makeset{T_{\sigma}}$ 
contradicting the fact that $q \in U^i_{\sigma}$.

Fix $n$ and let $\tau = \sigma^{\smallfrown}n$. Since $\tau$ properly
extends $\sigma$, we have that $q^{\smallfrown}r_0^{\smallfrown}G_i
\in \bigcap_m D_{m,\tau}$ by the comments two paragraphs above. 
Since $U_{\tau} \cup V_{\tau}$ is dense but 
$U_\tau \cap [q\concat r_0]=\emptyset$, it follows that
$V_{\tau}$ is dense in $[q^{\smallfrown}r_0]$ and therefore
$q^{\smallfrown}r_0^{\smallfrown}G_i \in V_{\tau}$.  From
$q^{\smallfrown}r_0^{\smallfrown}G_i \in \bigcap_m D_{m,\tau}$ and
$q^{\smallfrown}r_0^{\smallfrown}G_i \in V_{\tau}$, it follows that
$q^{\smallfrown}r_0^{\smallfrown}G_i \not \in |T_{\tau}|$ as required to
complete the contradiction.

To prove ($\Leftarrow$) in (2), assume that $\bigcup_n
U_{\sigma^{\smallfrown}n}$ is dense in $[q]$. We need to show that $q
\in U^i_{\sigma}$.  Fix $r_0 \in 2^{<\omega}$.  Since $\bigcup_n
U_{\sigma^{\smallfrown}n}$ is dense in $[q]$, it is also dense in
$[q^{\smallfrown}r_0]$. By the induction hypothesis and Proposition
\ref{prop:9}, $\bigcup_n U_{\sigma^{\smallfrown}n}$ is
$\Delta^1_1$. Let $A = \{ \tau : \exists n \, (
q^{\smallfrown}r_0^{\smallfrown}\tau \in U_{\sigma^{\smallfrown}n} )
\}$. $A$ is dense and is $\Delta^1_1$. Therefore, $G_i$ meets the set
$A$. Fix $\tau \in A$ such that $\tau \prec G_i$ and fix $n$ such
that $q^{\smallfrown}r_0^{\smallfrown}\tau \in
U_{\sigma^{\smallfrown}n}$.  Then $q^{\smallfrown}r_0^{\smallfrown}G_i
\in U_{\sigma^{\smallfrown}n}$.  So, as noted above,
$q^{\smallfrown}r_0^{\smallfrown}G_i \in \bigcap_m
D_{m,\sigma^{\smallfrown}n}$ and so
$q^{\smallfrown}r_0^{\smallfrown}G_i \in
|T_{\sigma^{\smallfrown}n}|$. As $r_0$ was arbitrary, this shows that $q
\in U_\sigma^i$.

The exact same argument shows that (2) is also satisfied when $i$ is 
replaced by $j$.  Therefore, $U_\sigma^i$ and $U_\sigma^j$ are 
described by exactly the same condition, so they are equal.

Finally, let $\sigma$ be an intersection node. 
First, consider the direction ($\Leftarrow$) of (3): 
Suppose that $q \not \in V^i_{\sigma}$ and fix $r_0$ such that 
$q^{\smallfrown}r_0 \in U^i_{\sigma}$. 
We will show that $q^{\smallfrown}r_0 \in U_{\sigma^{\smallfrown}n}$ for all $n$,
so $\bigcup_n V_{\sigma^{\smallfrown}n}$ is not dense in $[q]$ 
(it is disjoint from $[q\concat r_0]$). 

Fixing $n$, consider an arbitrary string $p$ extending $q^{\smallfrown}r_0$. 
Since $q^{\smallfrown}r_0 \in U^i_{\sigma}$, we know that $p^{\smallfrown}G_i \in \makeset{T_{\sigma}}$. Since $\sigma$ is an intersection node, it follows that 
$p^{\smallfrown}G_i \in \makeset{T_{\sigma^{\smallfrown}n}}$. Since $p$ was an arbitrary string extending $q^{\smallfrown}r_0$, this implies 
$q^{\smallfrown}r_0 \in U_{\sigma^{\smallfrown}n}$ as required to complete this direction of (3). 

To prove ($\Rightarrow$) in (3), assume $\bigcup_n V_{\sigma^{\smallfrown}n}$ 
is not dense in $[q]$. We need to show that $q \not \in V^i_{\sigma}$. 
Fix $r_0$ such that 
$[q^{\smallfrown}r_0] \cap \bigcup_n V_{\sigma^{\smallfrown}n} = \emptyset$. 
Therefore, for each $n$, $U_{\sigma^{\smallfrown}n}$ is dense in $[q^{\smallfrown}r_0]$. 

Fix an arbitrary string $p$ extending $q^{\smallfrown}r_0$. We claim that 
for all $n$, we have 
$p^{\smallfrown}G_i \in U_{\sigma^{\smallfrown}n}$. First, note that 
$U_{\sigma^{\smallfrown}n}$ is dense in $[p]$ and that by the 
induction hypothesis and Proposition \ref{prop:9}, 
$U_{\sigma^{\smallfrown}n}$ is $\Delta^1_1$. 
We shift $U_{\sigma^{\smallfrown}n}$ to a set 
$A = \{ \tau : p^{\smallfrown}\tau \in U_{\sigma^{\smallfrown}n} \}$ 
which is dense and $\Delta^1_1$, so $G_i$ meets $A$.
Let $\tau \in A$ be such that $\tau \prec G_i$. 
Then, $p^{\smallfrown}\tau \in 
U_{\sigma^{\smallfrown}n}$ and so $p^{\smallfrown}G_i \in U_{\sigma^{\smallfrown}n}$. Furthermore, as noted above, since $p^{\smallfrown}G_i \in 
\bigcap_m D_{m, \sigma^{\smallfrown}n}$, it follows that 
$p^{\smallfrown}G_i \in \makeset{T_{\sigma^{\smallfrown}n}}$. 
Since this property holds for each $n$ and since 
$\sigma$ is an intersection node, it follows that 
$p^{\smallfrown}G_i \in \makeset{T_{\sigma}}$. 
The string $p$ extending $q^{\smallfrown}r_0$ was arbitrary, 
so by the definition of 
$U^i_{\sigma}$, we have $q^{\smallfrown}r_0 \in U^i_{\sigma}$, 
and therefore $q \not \in V^i_{\sigma}$ to complete the proof of (3). 

We have actually proved a little more.  Inspecting the argument 
for $(\Rightarrow)$ in (3), we see that whenever 
$[q] \cap \bigcup_n V_{\sigma\concat n} = \emptyset$, we have $q \in U_\sigma^i$;
and inspecting the argument for $(\Leftarrow)$ in (3), we see 
that whenever $q \in U_\sigma^i$, we have 
$[q]\cap \bigcup_n V_{\sigma\concat n} = \emptyset$.  This gives a 
definition of $U_\sigma^i$ that does not depend on $i$, 
and indeed the arguments above could be repeated exactly for $U_\sigma^j$.
Therefore, $U_\sigma^i = U_\sigma^j$ in the case where $\sigma$ 
is an intersection as well.

We conclude that $(U_\sigma, V_\sigma)_{\sigma \in T}$ is a Baire 
decomposition for $T$, and so $T$ has a Baire approximation in $M$.
Therefore $\mathcal M$ satisfies $\DPB$ but not $\ATR$.
\end{proof}

\section{Decorating trees}

In order to show that $\DPB$ is strictly stronger than
$\LwCA$, we need to make some techniques for 
building non-standard 
Borel codes in a way that ensures they are completely determined.

A non-standard Borel code is a code that is not actually well-founded, 
but which the model thinks is well-founded.  These fake codes 
are essential for the strength of $\DPB$.  If a Borel code is 
truly well-founded, then it has a Baire code which is hyperarithmetic 
in itself.  Since any $\omega$-model of $\LwCA$ is 
closed under hyperarithmetic reduction, 
$\LwCA$ alone would be enough to guarantee the 
Baire code exists in the case when the Borel code is truly 
well-founded (at least in $\omega$-models). 
So now we are going to describe how to construct a
non-standard Borel code which makes every effort to be completely determined.

If we make a Borel code $T$ which is not well-founded, 
the most likely scenario is that it is also not completely determined. 
This is because, in general, it might take a jump hierarchy 
the height of the rank of $T$ in order to produce an 
evaluation map.  So in this section, we show how to 
add ``decorations'' to the tree, which shortcut the logic of 
the tree to make sure that for 
a small set of $X$, there is an evaluation map for $X$ in 
the decorated tree.  In Section \ref{sec:dpb_no_hyp}, 
``small'' is countable, and in Section \ref{sec:dpb_implies_generics}, 
``small'' is meager.  This comes at the cost of 
trashing any information about whether $X$ was in the
original set,
but if that set had a Baire approximation, 
then its decorated version should have the same Baire 
approximation, since the set of $X$ whose membership 
facts were overwritten is small.  We use this to show that
if the model satisfies 
$\DPB$, then 
the ``small'' set cannot be the entire second-order 
part of the model.

Suppose that we have a partial computable function 
$h$ which maps a number $b \in \kO^\ast$ to a pair
of $b$-ranked labeled trees $(P_b,N_b)$. 
We 
do not mind if $h$ happens to also make some 
outputs
for $b \not\in \kO^\ast$.

The intention is that when $b \in \kO$, any 
$X \in \makeset{P_b} \cup \makeset{N_b}$ will have 
an approximately 
$H_b^X$-computable evaluation map in the decorated tree, 
and $X$ will be in the decorated tree if $X \in \makeset{P_b}$ 
and out of the decorated tree if $X \in \makeset{N_b}$. 
(In practice we will always have 
$\makeset{P_b} \cap \makeset{N_b} = \emptyset$.)

The operation $\Decorate$ is defined below 
using effective transfinite recursion (with parameter 
$<_\ast$; see comment in the next paragraph), 
and therefore is well-defined 
on $a$-ranked trees $T$ for all $a \in \kO^{\ast,T}$.
This is because the effective transfinite recursion can be
carried out in $HYP(T)$ with the same result.

Note that as it is defined here, $\Decorate$ is not quite a
computable operation. That is because the relation $<_\ast$ 
is only c.e., not computable.  To make $\Decorate$ computable,
one should replace $\langle 2b+1\rangle$ below with 
$\langle 2\langle b,s\rangle +1\rangle$, where $s$ is the stage 
at which we learn that $b <_\ast \rho_T(\lambda)$.  This 
has no effect on the logic of the tree, but does result in 
excessive notational clutter.  The reader who prefers a computable
operation could replace $\langle 2b+1\rangle$ everywhere 
with the more complicated expression above.  For our purposes, 
it is perfectly fine that $\Decorate$ is computable relative to 
the parameter $<_\ast$ (a linear order which is itself 
$\emptyset'$-computable).  In any case, all results 
of this section do relativize and 
they will later be used in a relativized form.

\begin{definition} The operation $\Decorate$ is defined as 
follows.  The
inputs are an $a$-ranked labeled tree $T$
and a partial computable 
function $h$ as above.  
\begin{align*}
\Decorate(T, h) = \{\lambda\} &\cup 
\bigcup_{\langle n \rangle \in T} 
\langle 2n \rangle\concat \Decorate(T_{\langle n \rangle}, h)\\
&\cup \bigcup_{b <_\ast \rho_T(\lambda)}
\langle 2b+1\rangle\concat\Decorate(Q_b,h)
\end{align*}
where $Q_b = P_b$ if $\lambda$ is a $\cup$ in $T$, and 
$Q_b = N_b^c$ if $\lambda$ is a $\cap$ in $T$.

The rank and label of $\lambda$ in $\Decorate(T,h)$ 
are defined to coincide with the rank and label of 
$\lambda$ in $T$.  The ranks and labels of the other nodes in 
$\Decorate(T,h)$ are inherited from $\Decorate(T_{\langle n\rangle},h)$
or $\Decorate(Q_b,h)$ as appropriate.
\end{definition}

Since $P_b$ and $N_b$ are $b$-ranked, 
$\Decorate(T,h)$ satisfies the local requirements 
on a ranking.  So if $T$ is $a$-ranked, so is 
$\Decorate(T,h)$.

Similarly, if $T$ and each $P_b$ and $N_b$ are 
alternating, and  each $P_b$ and $N_b$ have
an intersection or leaf at their root, 
then $\Decorate(T,h)$ will also be 
alternating.  (Note that in this case, $N_b^c$ 
has a union at its root).

The following is the essential feature of a decorated tree.

\begin{proposition}\label{prop:essential_decoration}
If $\sigma \in \Decorate(T,h)$ has rank $b$, then 
for all $d <_\ast b$, 
$$\Decorate(T,h)_{\sigma \concat \langle 2d+1\rangle} = 
\Decorate(Q_d,h),$$
where $Q_d = P_d$ or $N_d^c$ as appropriate.
\end{proposition}
\begin{proof} By induction on the length of $\sigma$,
if $\sigma \in \Decorate(T,h)$, then there is some tree 
$S$ such that $\Decorate(T,h)_\sigma = \Decorate(S,h)$.
The rank of $\sigma$ in $\Decorate(T,h)$ 
coincides with the rank of $\lambda$ in $S$,
and this rank is $b$.  
Therefore, by the definition of $\Decorate(S,h)$, we have 
$\Decorate(S,h)_{\langle 2d+1\rangle} = \Decorate(Q_d,h)$.
\end{proof}

\begin{definition}
A \emph{nice decoration generator} is 
a partial computable function which 
maps any $b\in \kO^\ast$ to 
alternating, $b$-ranked 
trees $(P_b,N_b)$, where each 
$P_b$ and $N_b$ have
an intersection or a leaf at their root.  
\end{definition}

\begin{lemma}\label{lem:kO_induction}
Let $h$ be a nice decoration generator. 
Suppose $b \in \kO$, and suppose that 
$X \not\in \makeset{P_d} \cup \makeset{N_d}$
for any $d <_\ast b$.  Then for any $b$-ranked tree $T$,
$X \in \makeset{\Decorate(T,h)}$ if and only if $X \in 
\makeset{T}$.
\end{lemma}
\begin{proof}
By induction on $b$.  Since $b\in\kO$, $T$ 
is truly well-founded, so there is a unique evaluation map 
$f$ for $X$ in $T$.  
Further, for each $d<_\ast b$, there are unique evaluation 
maps $g_{P,d},g_{N,d}$ for 
$X$ in $\Decorate(P_d,h)$ and $\Decorate(N_d^c,h)$. 
Consider the function $g:\Decorate(T,h) \rightarrow \{0,1\}$ 
defined by 
$$g(\sigma)
= \begin{cases} 
f(\frac{\sigma}{2}) & \text{ if each component of $\sigma$ is even}\\
g_{Q,d}(\sigma_1) & \text{ if $\sigma = \sigma_0\concat \langle 2d+1\rangle\concat \sigma_1$ 
and each component of $\sigma_0$ is even,}\end{cases}$$
where the division $\sigma/2$ is taken componentwise, 
and where $Q$ is either $P$ or $N$ depending on whether 
$\sigma_0$ is a union or intersection in $\Decorate(T,h)$.

Since 
$g(\lambda) = f(\lambda)$,
it is enough to show 
that $g$ is an evaluation map for $X$ in
$\Decorate(T,h)$.   Clearly $g$ satisfies the 
logic of the tree at leaves and at nodes 
which have an odd component. 
Consider $\sigma \in \Decorate(T,h)$ where $\sigma$ is a $\cup$
and all components of $\sigma$ are even.
By induction, since $P_d$ is a $d$-ranked tree, 
$X \in \makeset{\Decorate(P_d,h)}$ if and only if $X \in 
\makeset{P_d}$.  By hypothesis, $X \not\in \makeset{P_d}$, 
so $g_{P,d}(\lambda) = 0$, so by Proposition \ref{prop:essential_decoration},
$g(\sigma\concat\langle 2d+1\rangle) = 0$. 
Therefore, the nodes of this form can be ignored: we have 
$$\exists m (g(\sigma\concat m) = 1) \iff 
\exists n (g(\sigma \concat \langle 2n \rangle) = 1) \iff f(\sigma/2) = 1$$
so $g(\sigma)$ takes the correct value.
The argument if 
$\sigma$ is a $\cap$ is similar, except that as 
$X \not\in \makeset{N_b}$, we have $X \in \makeset{N_b^c}$, 
and therefore $g_{N,d}(\lambda) = 1$, 
meaning that nodes of the form 
$\sigma\concat\langle 2d+1\rangle$ 
can be safely ignored when taking an intersection.
\end{proof}

\begin{lemma}\label{lem:2} Let $a \in \kO^\ast$ and $b \in \kO$ 
with $b<_\ast a$.
Let $T$ be an alternating, $a$-ranked tree and 
let $h$ be a nice decoration generator.
Suppose $X \in \makeset{P_b} \cup \makeset{N_b}$.  Then
\begin{enumerate}
\item $X$ has a unique evaluation
map in $\Decorate(T,h)$.
\item This evaluation map is $H_{b + O(1)}^{X\oplus T}$-computable.
\item If $b$ is $<_\ast$-minimal such that 
$X \in\makeset{P_b} \cup \makeset{N_b}$, and  
$b <_\ast \rho_T(\langle n \rangle)$ for all 
$\langle n \rangle \in T$, and $g$ is the unique 
evaluation map for $X$ in $\Decorate(T,h)$, then
\begin{enumerate}
\item $X \in \makeset{P_b} \setminus \makeset{N_b} \implies
g(\lambda) = 1$
\item $X \in \makeset{N_b} \setminus \makeset{P_b} \implies
g(\lambda) = 0$.
\end{enumerate}
\end{enumerate}
\end{lemma}
\begin{proof}
It suffices to show all three parts in the case when $b$ is
$<_*$-minimal such that $X \in |P_b| \cup |N_b|$.

We prove (1) and (2) by showing that for each $\sigma \in
\text{Decorate}(T,h)$, there is only one possible value for
$g(\sigma)$ for any evaluation map $g$ for $X$ in
$\text{Decorate}(T,h)$ and that $H_{b+O(1)}^{X \oplus T}$ suffices to
compute this value. Since these unique values satisfy the internal
logic of the tree (which the reader can verify from the description
below), they constitute an evaluation function for $X$ in
$\text{Decorate}(T,h)$, proving (1) and (2).

To show that there is only one possible value for $g(\sigma)$, we
break into cases depending on the rank and label of $\sigma$ in
$\text{Decorate}(T,h)$ and on whether $X \in |P_b|$ or $X \in
|N_b|$. Note that $H_b^{X \oplus T}$ can uniformly determine the
appropriate case for each $\sigma$.

Case 1. Suppose $\rho(\sigma) \leq_{*} b$. Since $b \in \mathcal{O}$,
$\text{Decorate}(T,h)_{\sigma}$ is truly well-founded. Therefore,
there is a unique evaluation map $f$ for $X$ in
$\text{Decorate}(T,h)_{\sigma}$ and we have $g(\sigma) =
f(\lambda)$. The map $f$ is uniformly $H_{b+O(1)}^{X \oplus T}$-computable.

Case 2. Suppose $b <_{*} \rho(\sigma)$, $\sigma$ is a union node in
$\text{Decorate}(T,h)$ and $X \in |P_b|$. In this case, we claim that
$g(\sigma)=1$.  By Proposition 4.2, all nodes extending
$\sigma^{\smallfrown}\langle 2b+1 \rangle$ have rank $b$ or
less. Therefore, there is a unique evaluation map $f$ on
$\text{Decorate}(T,h)_{\sigma^{\smallfrown}\langle 2b+1 \rangle}$ and
so $g(\sigma^{\smallfrown}\langle 2b+1 \rangle)=f(\lambda)$. By Lemma
4.4, $X \in |P_b|$ implies $f(\lambda) = 1$. Therefore,
$g(\sigma^{\smallfrown}\langle 2b+1 \rangle)=1$ and because $\sigma$
is a union node, $g(\sigma)=1$.

Case 3. Suppose $b <_{*} \rho(\sigma)$, $\sigma$ is an intersection
node in $\text{Decorate}(T,h)$ and $X \in |P_b|$. Since
$\text{Decorate}(T,h)$ is alternating, each node
$\sigma^{\smallfrown}m$ is either a union node or a leaf. If
$\rho(\sigma^{\smallfrown}m) \leq_{*} b$, then the value of
$g(\sigma^{\smallfrown}m)$ is fixed as in Case 1. If $b <_{*}
\rho(\sigma^{\smallfrown}m)$, then $g(\sigma^{\smallfrown}m) = 1$ as
in Case 2. Together, these values determine $g(\sigma)$ uniquely.
$H_{b}^{X \oplus T}$ suffices to compute the values of
$g(\sigma^{\smallfrown}m)$ and it takes one extra jump to determine if
$g(\sigma^{\smallfrown}m) = 1$ for all $m$, and hence determine
$g(\sigma)$.

Case 4. Suppose $b <_{*} \rho(\sigma)$, $\sigma$ is an intersection
node in $\text{Decorate}(T,h)$ and $X \in |N_b|$. An analogous
argument to Case 2 shows that $g(\sigma) = 0$.

Case 5. Suppose $b <_{*} \rho(\sigma)$, $\sigma$ is a union node in
$\text{Decorate}(T,h)$ and $X \in |N_b|$. This case is analogous to
Case 3 and the unique value of $g(\sigma)$ can be determined with one
extra jump.

These cases are exhaustive, but if $\makeset{P_b} \cap \makeset{N_b} \neq \emptyset$, then
more than one case can apply. However, if $X \in \makeset{P_b} \cap \makeset{N_b}$, the
cases are compatible. In this
degenerate situation, we have that for any $\sigma$ such that $b <_*
\rho(\sigma)$, $g(\sigma) = 1$ if $\sigma$ is a union node and
$g(\sigma) = 0$ if $\sigma$ is an intersection node. This completes
the proof of (1) and (2).

For (3), if $X \in \makeset{P_b} \setminus \makeset{N_b}$, 
and if $\lambda$ is $\cup$, then $g(\lambda)=1$ 
just as above.  But if $\lambda$ is $\cap$, then we claim that 
for each $m$, $g(\langle m \rangle) = 1$.  
(Note that neither $\lambda$ nor $\langle m \rangle$ can be 
a leaf in $T$ because $b<_\ast a$ and the hypothesis on part (3) specifies that 
$b <_\ast \rho_T(\langle m\rangle)$ for each $m$).
If $m =2n$ 
for some $\langle n \rangle \in T$, or if $m = 2d+1$ 
for some $d>_\ast b$, then because
$b<_\ast \rho_T(\langle n \rangle)$ for all $n$,
and each $\langle m \rangle$ is a union, again we have
$g(\langle m \rangle) = 1$ for such $m$.
In the remaining case, when $m = 2d+1$ with $d \leq_\ast b$, 
then since $b$ is minimal such that 
$X \in \makeset{P_b} \cup \makeset{N_b}$, and 
$X \not\in \makeset{N_b}$, we have $X \in \makeset{N_d^c}$. 
So by Lemma \ref{lem:kO_induction}, $X \in \makeset{\Decorate(N_d^c, h)}$, 
so $g(\langle 2d+1\rangle) = 1$.
Since $g(\langle m \rangle) = 1$ for all $m$,
we have $g(\lambda) =1$ as well.
A complementary argument establishes (3b).\end{proof}

\section{$\DPB$ does not hold in $HYP$}\label{sec:dpb_no_hyp}

We now show that $\DPB$ is not a theory of hyperarithmetic analysis
by showing that $\DPB$ fails in the $\omega$-model $HYP$.
In brief, we let $E_a$ code a canonical universal $\Sigma^0_a$ set.
Applying this definition also to pseudo-ordinals $a^\ast$, 
we make a computable code for the set $$\bigcup_{b<_\ast a^\ast} \makeset{E_b} \cap \{X : b \text{ is least s.t. } X \leq_T H_b\}.$$
We decorate the code to give each $H_b$-computable
    set an $H_b$-computable evaluation map.  Then we argue that 
the result is a code which $HYP$ thinks is well-founded and 
completely determined, but which can have no $HYP$ Baire approximation.

 \begin{theorem}\label{thm:dpb_fails_in_hyp}
$\DPB$ does not hold in $HYP$.
\end{theorem}
\begin{proof} Using Proposition \ref{prop:Hformula},
there is a computable procedure which, on inputs $a \in \kO$,
$e\in \mathbb N$, $p \in 2^{<\omega}$,
outputs an index for a $2^a$-ranked
computable $L_{\omega_1,\omega}$ 
formula $F_{a,e,p}$, which holds true if and only if $p \in W^{H_a}_e$. 
Transform each formula $F_{a,e,p}$ into a Borel code by 
swapping {\tt false} 
for $\emptyset$, and {\tt true} for $[0^e1\concat p]$.  
Then take the 
union of all of these, obtaining a code $E_a$ of rank $a + O(1)$
such that for all $a \in \kO$,
$$\makeset {E_a}  = \bigcup_{e,p \ :\ p \in W_e^{H_a}} [0^e1\concat p].$$
For any pseudo-ordinal $a^\ast$,
$W_{p(a^\ast)}$ is not well-founded, but it
has no hyperarithemtic descending sequence, so
$HYP$ believes $W_{p(a^\ast)}$ is well-founded. Then 
$HYP$ also believes that $E_{b}$ is well-founded 
for any $b <_\ast a^\ast$, because $E_b$ is $(b+O(1))$-ranked,
so any path through $E_b$ would reveal a 
descending sequence in $W_{p(a^\ast)}$.  
We may assume 
that $E_{b}$ are alternating 
and $(b+O(1))$-ranked for all $b \leq_\ast a^\ast$.  For the 
sake of a later application of Lemma \ref{lem:2}, 
note that we can also assume that the rank of 
$E_{a^\ast}$ is a
successor, so of the form $2^x$ for some $x$, 
and that for each $\langle n \rangle \in E_{a^\ast}$, 
the rank of $\langle n \rangle$ in $E_{a^\ast}$ is $x$.

Similarly, there is a computable procedure which, for each 
$b \in \kO$, outputs a $(b+O(1))$-ranked 
Borel code $S_b$ such that
$$\makeset{S_b} = \{ X \in 2^\omega : X \leq_T H_b 
\text{ and for all } c<_\ast b, X \not\leq_T H_c\}.$$
We think of $S_b$ as coding a slice of $HYP$.
Just as for $E_b$, we have that 
for any $b<_\ast a^\ast$, $HYP$ thinks that $S_b$ is 
well-founded.

For each $b<_\ast a^\ast$, define $P^b$ and $N^b$ so that
they are alternating, and
$$\makeset{P^b} = \makeset{S_b} \cap \makeset{E_b}, \qquad
\makeset{N^b} = \makeset{S_b} \cap \makeset{E_b^c}.$$

Observe that $P^b$ and $N^b$ can be both $(b+ k)$-ranked,  
where $k$ is some fixed finite ordinal.  Let 
$h$ be the function which, on input $b$, outputs 
$P_b = P^{b-k}$ and $N_b = N^{b-k}$ if the operation 
$b-k$ can be performed, and outputs a degenerate $b$-ranked 
tree coding the empty set, if $b$ is 
less than $k$ successors from a limit ordinal.

We claim that $\Decorate(E_{a^\ast},h)$ is completely determined in $HYP$.
Observe that $h$ is a nice decoration generator.  
Let $X \in HYP$.  Then there is some $b \in \kO$
with $b<_\ast a$ 
such that $X \leq_T H_b$.  Since $a^\ast$ is a pseudo-ordinal, 
$b+ O(1) <_\ast a^\ast$ 
is satisfied.  By the choice of $b$ we have 
$X \in \makeset{S_b} = \makeset{P_{b+k}}\cup\makeset{N_{b+k}}$. 
Therefore, by Lemma \ref{lem:2}, $X$ has a $HYP$ evaluation 
map.  Therefore, $\Decorate(E_{a^\ast}, h)$ is 
completely determined in $HYP$.

Suppose for contradiction that $\Decorate(E_{a^\ast}, h)$ has a 
$HYP$ Baire approximation. Let $b\in \kO$ with $b<_\ast a^\ast$
and with the Baire approximation $(U,V,\{D_n\}_{n\in\omega})\leq_T H_b$.
By the recursion theorem, there is an index $e$ 
such that 
$$W_e^{H_b} = \{p : 0^e1\concat p \in V\}$$
where $H_b$ is used to compute $V$.  
Choose $p$ with 
$0^e1\concat p \in U \cup V$, 
this is possible as $U \cup V$ is dense. 
Let $X \in HYP$ 
be such that 
\begin{enumerate}
\item $0^e1\concat p \prec X$ 
\item $X \leq_T H_b$ but $X\not\leq_T H_c$ for any $c<_\ast b$, 
\item $X \in D_n$ for all $n$.
\end{enumerate}
This is possible because the $D_n$, and the dense sets which 
need to be met to avoid being computed by $H_c$ for 
$c<_\ast b$, are uniformly $H_b$-computable.

Now $b + k$ 
is least such that $X \in \makeset{P_{b+k}}\cup \makeset{N_{b+k}} 
=\makeset{S_b}$.  By Lemma \ref{lem:2}, 
$X \in \makeset{\Decorate(E_{a^\ast},h)}$ if and only if 
$X \in \makeset{E_b}$.  Because $X$ meets each $D_n$ and 
$U \cup V$, by the definition of a Baire code, 
we have $X \in \makeset{\Decorate(E_{a^\ast},h)}$ if and only if 
$X \in U$.  To establish the contradiction, it suffices to 
show that $X \in \makeset{E_b}$ if and only if $X \in V$.

Observe $X \in \makeset{E_b}$, if and only if,
for some $q$ extending $p$, we have 
$0^e1\concat q \prec X$ and
$q \in W_e^{H_b}$.  But this happens 
if and only if for some such $q$, we have
$0^e1\concat q \in V$.
\end{proof}

  \section{$\DPB$ implies $HYP$ generics exist in $\omega$-models}\label{sec:dpb_implies_generics}

The next theorem shows that $\DPB$ implies the existence of hyperarithmetic
  generics in $\omega$-models.  
In short, if $\mathcal M$ has $Z$ but no $\Delta^1_1(Z)$-generics,
    there is a pseudo-ordinal $a^\ast$ which $\mathcal M$ thinks is well-founded.
This pseudo-ordinal can be used to construct a code for the 
following subset of $M$, where $E_b$ denotes a code for a universal $\Sigma^Z_b$ set:
    $$\bigcup_{b<_\ast a^\ast} \makeset{E_b} \cap \{X : b \text{ is least s.t. $X$ is not generic relative to } H_b^Z\}$$
After decorating this code, it becomes completely determined for every
non-$\Delta^1_1(Z)$-generic.  If this code has a Baire decomposition, 
meeting the associated dense sets creates a $\Delta^1_1(Z)$-generic.

\begin{theorem}\label{thm:dpb_implies_generics}
If $\mathcal M$ is an $\omega$-model which satisfies
  $\DPB$, then for every $Z\in \mathcal M$, there is a $G \in \mathcal M$
  such that $G$ is $\Delta^1_1$-generic relative to $Z$.
  \end{theorem}
\begin{proof}
Let $M$ be the second-order part of an $\omega$-model which satisfies 
$\DPB$.  Then by Proposition \ref{prop:dpb_implies_lwca}, whenever $Z \in M$, 
we also have that $H_b^Z \in M$ for every $b \in \mathcal O^Z$.

Case 1: Suppose $\mathcal M$ is a $\beta$-model (that is, for every 
tree $T \in M$, if $\mathcal M \models ``T \text{ is well-founded}$'', 
then $T$ is truly well-founded.)  Let $Z \in M$.
Because $\{G : G \text{ is $\Delta^1_1(Z)$-generic}\}$ 
is a $\Sigma^1_1(Z)$ set, the $Z$-computable tree corresponding to the 
$\Sigma^1_1(Z)$ statement ``there is a $\Delta^1_1(Z)$-generic'' 
has a path in $M$, and that path computes a $\Delta^1_1(Z)$-generic 
$G$. 
Therefore, the theorem holds when $\mathcal M$ is a $\beta$-model.

Case 2: Suppose that there is some tree $S \in M$ which $\mathcal M$ 
believes is well-founded, but in reality is ill-founded.  
Let $Z \in M$, and without loss of generality assume that 
$Z \geq_T S$ (without this assumption we find a 
$\Delta^1_1(Z\oplus S)$-generic $G$, but such $G$ is also
$\Delta^1_1(Z)$-generic.)
By Proposition \ref{prop:LObound}, there is a $Z$-computable function 
which, given the index of a truly 
well-founded $Z$-computable linear order, outputs an 
element of $\kO^Z$ which bounds its order type.
Applying that function to the Kleene-Brouwer ordering on $S$ produces a 
pseudo-ordinal $a^\ast \in \kO^{\ast,Z}$ such that $W^Z_{p(a^\ast)}$ is 
not truly well-founded, but it has no descending sequence in $M$.

Relativize the definitions of $<_\ast$, ranked trees, 
$\Decorate$, and Lemmas \ref{lem:kO_induction} and \ref{lem:2} to $Z$.
Note that because $M$ is hyperarithmetically closed, all the evaluation
maps provided by relativized versions of Lemmas \ref{lem:kO_induction}
and \ref{lem:2} are in $M$.

As in the previous theorem, 
there is a $Z$-computable procedure which maps any $b \in \kO^Z$ 
to an alternating code $E_b$ of rank $b + O(1)$ such that 
$$\makeset{E_b} = \bigcup_{e,r\ :\ r \in W_e^{H_b^Z}} [0^e1\concat r].$$
Further, using Proposition \ref{prop:Hformula},
there are $Z$-computable procedures which map 
each $b \in \kO^Z$ to a code $S_b$ of rank $b + O(1)$ such that
\begin{multline*}
\makeset{S_b} = \{X \in 2^\omega : \text{$X$ is not 1-generic relative 
to $H_b^Z$,}\\ \text{ but for all $c<_\ast^Z b$, $X$ is 1-generic relative to $H_c^Z$}\},\end{multline*}
and alternating codes $P_b$ and $N_b$ of rank $b$ such that 
$$\makeset{P_{b}} = \makeset{S_{b-O(1)}} \cap \makeset{E_{b-O(1)}}, \qquad
\makeset{N_{b}} = \makeset{S_{b-O(1)}} \cap \makeset{E_{b-O(1)}^c},$$
(and for $b$ that are within $O(1)$ of a limit ordinal, $P_b$ and 
$N_b$ are degenerate 
$b$-ranked trees coding the empty set as before).  

Let us be 
a little more specific and say that the code for $P_b$ is
made exactly as one would expect: it is $\operatorname{Alternate}(P_b')$,
where 
$$P_b' = \{\lambda\} \cup \langle 0 \rangle \concat S_{b-O(1)} \cup \langle 1 \rangle \concat E_{b-O(1)},$$
the root $\lambda$ is a $\cap$ of rank $b$ in $P_b'$, and all other ranks and labels
are inherited from their respective subtrees.  We remark that because 
the root of $P_b'$ is a $\cap$, the root of $E_{b+O(1)}$ is a $\cup$,
and $E_{b+O(1)}$ is already alternating, we have $(P_b)_{\langle 1\rangle} = E_{b-O(1)}$.

Because the outputs of Proposition \ref{prop:Hformula} are well-defined for
all $b \in \kO^Z$, so also are the codes $P_b$ and $N_b$.  Also, for
any $b<_\ast^Z a^\ast$, $M$ believes these codes to be well-founded
because they are $b$-ranked.

Let $h$ be the name of the nice decorating 
function mapping $b$ to 
$(P_b,N_b)$, and consider the code $T := \Decorate^Z(E_{a^\ast},h)$.
Observe that since $\lambda$ in $E_{a^\ast}$ is a $\cup$, 
we know that $\lambda$ in $T$ is a $\cup$.

If $T$ is not completely determined, let $G \in M$ 
be such that $G$ does not have an evaluation map in $T$.  We claim that $G$ is 
$\Delta^1_1(Z)$-generic.
If $G$ is not $\Delta^1_1(Z)$-generic, then there is some 
least $b \in \kO^Z$ with $b<_\ast^Z a^\ast$ 
such that $G$ is not 1-generic relative 
to $H_b^Z$.  Then we would have $G \in \makeset{S_b}$, 
and therefore by Lemma \ref{lem:2}, $G$ would have an
evaluation map in $T$.

If $T$ is completely determined, then since $\mathcal M$ 
models $\DPB$,
let $(U_\sigma,V_\sigma)_{\sigma \in T} \in M$ 
be a Baire decomposition for $T$.  Let $\{D_i\}_{i<\omega} \in M$ 
be the associated sequence of dense sets as in Proposition 
\ref{prop:decomposition_implies_approximation}.  
For any $p \in 2^{<\omega}$, define 
$D_{i,p} = \{q : p\concat q \in D_i\}$.
We claim that any $G \in \cap_{i,p} D_{i,p}$ 
is $\Delta^1_1(Z)$-generic.  For this we argue that
every dense open $B \in \Delta^1_1(Z)$ actually contains
$D_{i,p}$ for some $i,p$.
Let $b \in \kO^Z$ and $e$ be such that 
$B = W_e^{H_b^Z}$.  Then
$T_{\langle 2(b+O(1))+1\rangle} = \Decorate(P_{b+O(1)},h)$,
where $\makeset{P_{b+O(1)}} = |S_b| \cap |E_b|$.  Therefore, 
there is some $\sigma \in T$ such that 
$T_\sigma = \Decorate(E_b,h)$.  Since 
$E_b$ has a union at the root, this $\sigma$ is a union.
Let $p = 0^e1$.
We claim that $D_{\ell,p} \subseteq B$, where 
$D_\ell = \cup_m U_{\sigma\concat m}\cup V_{\sigma}$.
Let $q$ be such that $p\concat q \in D_\ell$.
To finish the proof, we need to show that $[q] \subseteq B$.

For the remainder of this proof, any $X$ which meets 
the following conditions
will be called \emph{sufficiently generic}:
\begin{itemize}
\item $X \in \cap_i D_i$, and
\item $X$ is 1-generic relative to $H_{b+O(1)}^Z$
\end{itemize}
Observe that for every $r \in 2^{<\omega}$, there is a 
sufficiently generic $X \in M$ with $r \prec X$.
Also, observe that for all such $X$ and all 
codes $R$ which are $c$-ranked for some $c\leq_\ast b +O(1)$,
the second condition implies that $c, X$ and $R$ 
satisfy the conditions of Lemma \ref{lem:kO_induction},
and so $X \in \makeset{\Decorate(R,h)}$ 
if and only if $X \in \makeset{R}$. Finally, 
by Proposition \ref{prop:decomposition_implies_approximation},
for all sufficiently generic $X$ and all $\tau \in T$,
we have $X \in \makeset{T_\tau}$ if and only if 
$X \in U_\tau$.

If $X$ is sufficiently generic and $p\concat q \prec X$, then 
$X \in p\concat B$, and so 
$X \in \makeset{E_b}$, and so by Lemma \ref{lem:kO_induction},
$X\in \makeset{\Decorate(E_b,h)} = \makeset{T_\sigma}$.
Therefore, it is impossible 
that $X \in V_\sigma$, so we conclude
$p\concat q \in U_{\sigma\concat m}$ for some $m$.  
Therefore, for sufficiently generic $X$ with 
$p\concat q \prec X$, we have 
$X \in \makeset{T_{\sigma\concat m}}$.

If $m = 2c + 1$ for some $c\leq_\ast b + O(1)$, then 
$T_{\sigma\concat m} = \Decorate(P_c,h)$.  But
for any sufficiently generic $X$, we have 
$X \not\in \makeset{P_c}$, so this case is impossible. 
Therefore, $m = 2n$ for some 
$\langle n \rangle \in E_b$.  It follows from the 
definition of $\Decorate$ that 
$T_{\sigma\concat m} = \Decorate((E_b)_{\langle n\rangle},h)$.
So for sufficiently generic $X$ with $p\concat q \prec X$, 
we have $X \in \makeset{(E_b)_{\langle n \rangle}}$.

Now we will use a property of the codes $E_b$ which 
follows from how they are defined at the beginning of the 
proof of Theorem \ref{thm:dpb_fails_in_hyp}.  
The code $E_b$ was obtained as the union of 
many codes $F_{b,e,r}$, at whose leaves the only 
options are $[0^e1\concat r]$ or $\emptyset$.  
The code $E_b$ was also post-processed so that it would 
be alternating, but while this process can 
break up the first-level subtrees $F_{b,e,r}$, it can never combine 
them together.  (See the discussion at the end of 
Section \ref{sec:breakapart} for details.)
Therefore, for every 
$\langle n \rangle \in E_b$, there is an $r$ such 
that whenever $\langle n \rangle\concat \tau \in E_b$ 
is a leaf, its attached clopen set is either 
$[0^e1\concat r]$ or $\emptyset$.  Fixing $r$ 
associated to $n = m/2$ for the $m$ found above, 
we observe that an evaluation map on $(E_b)_{\langle n\rangle}$ 
that works for 
one $Y \in [0^e1\concat r]$ works for all such $Y$, 
and we conclude that 
$\makeset{(E_b)_{\langle n\rangle}}$ is equal to either 
$\emptyset$ or $[0^e1\concat r]$.  
It must be the latter because
$X \in \makeset{(E_b)_{\langle n\rangle}}$ for all sufficiently 
generic $X$ with $p\concat q \prec X$.  It follows 
that $[r]\subseteq B$.  Furthermore, any sufficiently 
generic $X$ that does not extend $p\concat r$ must 
be out of $\makeset{(E_b)_{\langle n \rangle}}$, so it 
must be that $[q]\subseteq [r]$.
Therefore, $[q] \subseteq B$, as desired.
\end{proof}

\section{Application to the Borel dual Ramsey theorem}\label{sec:bdrt}

As an application of Theorem \ref{thm:1}, we identify a natural 
formulation of the Borel dual Ramsey 
theorem for 3 partitions and $\ell$ colors ($\BorelDRT^3_\ell$) 
as a principle which lies strictly below $\ATR$, but all of 
whose $\omega$-models are closed under hyperarithmetic reduction.

\begin{theorem}[Borel dual Ramsey theorem, \cite{carlsonsimpson1984}]
For every Borel $\ell$-coloring of the set of partitions of $\omega$ 
into exactly $k$ pieces, there is an infinite partition $p$ of $\omega$ 
and a color $i<\ell$ such that every way of coarsening $p$ 
down to exactly $k$ pieces is given color $i$.
\end{theorem}
Since the set of partitions of $\omega$ into exactly $k$ pieces 
can be coded naturally as a Borel subset of $k^\omega$, a natural 
way to formulate the hypotheses of the above theorem is roughly
``Whenever there are Borel codes $T_1,\dots T_\ell$ such that 
for every $X \in k^\omega$, we have $X \in \makeset{\cup_{i<\ell} T_i}$, 
...''
(See below for a precise formalization).

Therefore, the Borel dual Ramsey theorem has a natural formulation 
in terms of completely determined Borel sets.  In
\cite{ProemelVoigt1985, DFSW}, it was shown that a solution to 
$\BorelDRT^k_\ell$ can in general be obtained by a 
two-step process:
\begin{enumerate}
\item Use the fact that every Borel set has the property of Baire 
to come up with a Baire approximation to each color in the given coloring.
\item Apply a purely combinatorial principle $\mathsf{CDRT}^k_\ell$ 
to a coloring of $(k-1)^{<\omega}$ obtainable from 
the Baire approximation from (1).
\end{enumerate}
If we represent the coloring in the natural way described below, then 
$\DPB$ can be used to carry out (1).  It was known to Simpson 
(see \cite{DFSW}) that 
$\mathsf{CDRT}^3_\ell$ follows from Hindman's Theorem ($\mathsf{HT}$), 
which follows from $\ACA^+$ by \cite{BlassHirstSimpson1985}.  
Therefore, the following natural formalization of
$\BorelDRT^3_\ell$ follows from $\DPB + \ACA^+$. We first 
give the formalization of the space of $k$-partitions of 
$\omega$, and then the formalization of $\BorelDRT^3_\ell$.

\begin{definition}[Partitions of $\omega$, \cite{DFSW}]
In $\RCA$, a partition of $\omega$ into exactly $k$ pieces 
is a function $p \in k^\omega$ such that $p$ is surjective, 
and for each $i<k-1$, 
$$\min \{ n : p(n) = i\} < \min \{n : p(n) = i+1\}.$$
A partition of $\omega$ into infinitely many pieces 
is a surjective function $p \in \omega^\omega$ which 
satisfies the above condition for each $i \in \omega$.
\end{definition}

The set of partitions described above is an open subset of $k^\omega$ 
representable in $\RCA$ by a completely determined Borel code, as the 
reader can verify.  (For the case $k=3$, 
the set in question is the union of the 
sets $O_{a,b}$ introduced at the start of the proof of Theorem
\ref{thm:hypclosed}.)  Let $P_3$ denote this completely determined Borel code
in the case $k=3$.

\begin{definition}[Formal Borel dual Ramsey theorem for $3$ partitions 
and $\ell$ colors]
In $\RCA$, $\BorelDRT^3_\ell$ is the principle which states: 
Whenever $T_0,\dots T_{\ell-1}$ are Borel codes such that for all 
$X \in \makeset{P_3}$, we have $X \in \makeset{ \bigcup_{i<\ell} T_i}$, 
then there is an infinite partition $p$ of $\omega$ 
and a color $i<\ell$ such that whenever $X \in \makeset{P_3}$, 
$X \circ p \in \makeset{T_i}$.
\end{definition}

We would like to say the hypotheses of the theorem imply that the
$\{T_i\}_{i<\ell}$ are all completely determined.  This is not quite true 
(perhaps $X \in \makeset {T_i}$ is not completely determined for 
some $X \not\in \makeset{P_3}$).  However, a small modification
of the existing codes makes them completely determined.

\begin{lemma}
$(\ACA)$
Suppose that $S$ is a completely determined Borel code
and $T$ is a Borel code.  Suppose that for all $X \in \makeset{S}$,
there is an evaluation map for $X$ in $T$.  Then
there is a completely determined Borel code $R$ such that 
for all $X \in \makeset{S}$, we have
$$X \in \makeset{T} \iff X \in \makeset{R}.$$
\end{lemma}
\begin{proof}
Let $R$ be obtained from $T$ by replacing each leaf $\sigma$ in $T$ 
with the intersection of $S$ and the clopen set coded by $\ell(\sigma)$ in $T$.
If $X \in \makeset S$, then an evaluation map for $X$ in $R$ is 
obtained by starting with an evaluation map for $X$ in $T$ and then
filling in the evaluation map for $X \in \makeset{S}$ at all the places where 
$S$ appears in $R$.  If $X \in \makeset{S^c}$, an evaluation map for 
$X$ in $R$ is obtained by filling in all the original nodes of $T$ with 0, 
filling in the evaluation map for $X$ in $S$ at all the places where 
$S$ appears in $R$, and filling in the correct values on the remaining leaves
which were copied from $T$.

(Note: it does not work to let $R$ be simply the intersection of $S$ and $T$,
because the definition of completely determined requires that 
the entire evaluation map be filled out, even if most of it is not used.)
\end{proof}

It follows that if $(T_i)_{i<k}$ satisfy the hypotheses of the formal
Borel dual Ramsey theorem above, they can be taken to be completely 
determined without loss of generality.  Therefore, the discussion 
preceding the formal definitions proves that 
$\DPB + \ACA^+ \vdash \BorelDRT^3_\ell$ over $\RCA$.

The $\omega$-model which 
was constructed to prove Theorem \ref{thm:1} is closed 
under hyperarithmetic reduction, and therefore satisfies $\ACA^+$
as well as $\DPB$.  Therefore, $\BorelDRT^3_\ell$ 
holds in this model, while $\ATR$ does not.  This shows that 
the formulation of $\BorelDRT^3_\ell$ discussed 
here is strictly weaker than $\ATR$.

On the other hand, we have the following, which essentially follows 
from a more detailed version of the analysis in Section 4 of \cite{DFSW}.
\begin{theorem}\label{thm:hypclosed}  Let $\ell \in \omega$ with $\ell \geq 2$.
Every $\omega$-model of $\BorelDRT^3_\ell$ is closed 
under hyperarithmetic reduction.
\end{theorem}
\begin{proof} It suffices to consider the case $\ell = 2$.  We will 
first define some important subsets of $3^\omega$.
For each $a,b$ with $0<a<b$, let $O_{a,b}$ be the 
clopen set given by the finite collection of strings 
$$O_{a,b} = \{ \sigma \in 3^{b+1} : a = \min\{ n : \sigma(n) = 1\}
\text{ and } b = \min\{ n : \sigma(n) = 2\} \}$$
Then the set of partitions of $\omega$ into exactly 3 pieces is 
given by $P_3 = \bigcup_{0<a<b} O_{a,b}$.  

Let $M$ be the second-order part of an 
$\omega$-model $\mathcal M$ of $\BorelDRT^3_2$.  We first show 
that $\mathcal M$ satisfies $\ACA$.  Let $A \in M$.  Let $R$ 
be the following labeled Borel code.\footnote{We use 
standard computability-theoretic notation: for any 
$s \in \mathbb N$, let $A'_s$ denote 
$\{x < s : \Phi_{x,s}^A(x) \downarrow\}$, and for 
any $X$ let 
$X\uhr s$ denote the string $\sigma$ of length $s$ 
describing the characteristic function of $X$ on $\{0,\dots, s-1\}$.}
$$R = \bigcup_{0<a<b} \ \bigcap_{s>b} C_{a,b,s} \text{  where  }
C_{a,b,s} = \begin{cases} O_{a,b} & \text{ if } A'_b\uhr a = A'_s\uhr a\\
\emptyset & \text{otherwise.}\end{cases}$$
Then $R$ is completely determined.  For any $X \in 3^\omega$, there is 
at most one pair $a,b$ such that $X \in O_{a,b}$, so 
an evaluation map for $X$ in $R$ may safely put zeros 
at every node of $R$ except for the root and the nodes 
of the distinguished subtree $\cap_{s>b} C_{a,b,s}$.  The leaves of 
that subtree can be $X\oplus A$-computably filled out. 
Then the root of $R$ and the root of the subtree 
$\cap_{s>b} C_{a,b,s}$ may 
be non-uniformly supplied with their unique correct values.

Exactly as in the proof of \cite[Theorem 4.5]{DFSW}, we now
show that for any 
infinite partition $p$ of $\omega$ which is homogeneous 
for the coloring defined by $\makeset{R},\makeset{R^c}$, 
the principal function of $p$ dominates the least modulus 
function for $A'$.  For each $i$, let $p_i = \min \{n : p(n) = i\}$
(these are the minimum elements of the blocks of $p$).
First we claim that $p$ is homogeneous for color $R$.  Let 
$s$ be large enough that $A_s'\uhr p_1 = A' \uhr p_1$. 
Let $j$ be large enough that $p_j >s$.  Then the coarsening of 
$p$ which keeps blocks 1 and $j$, while collapsing all other blocks 
in with the zero block, is an element of $R$.  By similar reasoning,
but now looking at the 3-partition of $\omega$ obtained from
$p$ by keeping the only the $i$ and $(i+1)$ blocks
separate from the 0 block, we have $A_{p_{i+1}}'\uhr p_i = A' \uhr p_i$.
Thus $p \geq_T A'$.  
Therefore, $\mathcal M \models \ACA$.

Now suppose that $A \in M$ and $3\cdot 5^e \in \kO^A$.  Suppose that 
for all $d \leq_\kO 3\cdot 5^e$, we have $H_d^A \in M$.  Then 
we claim that $H_{3\cdot 5^e}^A \in M$.  By a result of Jockusch 
\cite{Jockusch1968} discussed in more detail below, 
the hyperarithmetic sets are exactly 
those that can be computed from sufficiently fast-growing 
functions.  As in \cite[Theorem 4.7]{DFSW}, we construct
a Borel coloring which forces any solution to $\BorelDRT^3_2$ 
to compute a sufficiently fast-growing function.  To prove the
associated Borel code is completely determined, we need a more detailed 
analysis than what was given in \cite{DFSW}.

More specifically, Jockusch's result has plenty of uniformity:
there are computable 
functions $h$ and $k$ such that for all $d \in \kO^A$, 
whenever $g:\omega\rightarrow \omega$ dominates the 
increasing function
$$f_d(n) := \Phi_{h(d)}^{H^A_d}(n),$$ we have 
$$\Phi_{k(d)}(A\oplus g) = H_d^A.$$
(To get this from the proof of
\cite[Theorem 6.8]{Jockusch1968}, apply 
\cite[Exercise 16-98]{Rogers_book} to conclude that 
the sets $H_d^A$ are in fact uniformly Turing equivalent to 
implicitly $\Pi^0_1(A)$-definable functions $f_d$.)

Uniformly in $d \in \kO^A$ and $a,b, \in \omega$, and $A$,
 there are Borel codes $C_{a,b,d}$ of well-founded rank 
$d + O(1)$ such that 
$$C_{a,b,d} = \begin{cases} O_{a,b} & \text{ if } b \geq f_d(a) \\
\emptyset & \text{ otherwise.}\end{cases}$$
The uniformity follows from the existence of $h$ above 
and the $A$-uniformity of producing a formula of $L_{\omega_1,\omega}$ 
to assess facts about $H_d^A$ (Proposition \ref{prop:Hformula}).

For each $n<\omega$, let $d_n = \Phi_e(n)$. 
Now let $R$ be the labeled Borel code 
$$R = \bigcup_{0<a<b} \ \bigcap_{i \leq a} C_{a,b,d_i}.$$
For any $X \in 3^\omega$, there is at most one 
pair of $a,b$ such that $X \in O_{a,b}$, so as above, 
any evaluation map for $X$ in $R$ can safely fill in zeros 
everywhere except for the root of $R$ and the distinguished 
subtree rooted at $\cap_{i\leq a}C_{a,b,d_i}$.  This subtree has well-founded 
rank $d_a + O(1)$, so the unique evaluation map on it is 
$H_{d_a + O(1)}^A$-computable.  Because $H_d^A \in M$ for all 
$d\leq_\kO 3\cdot 5^e$, this evaluation map exists in $M$.  
Therefore, $R$ is completely determined in $M$.

Now let $p\in M$ be any infinite partition of $\omega$ 
which is a solution to $\BorelDRT^3_2$ for 
the coloring $\makeset{R}, \makeset{R^c}$.  
Define, for each $i$,
$$p_i = \min\{n : p(n) = i\}.$$
Continuing to copy the proof of \cite[Theorem 4.5]{DFSW}, 
for every $0<s<t$, consider the coarsening $X_{s,t}$ of $p$
obtained by keeping the $s$ and $t$ blocks of $p$
and collapsing all other blocks to 0.  Since $t$ can be chosen 
arbitrarily large, for every $s$ there is a $t$ such that 
$$X_{s,t} \in \bigcap_{i\leq p_s} C_{p_s,p_t,d_i}$$
and therefore $P_3 \circ p$ is monochromatic for color $R$, 
and $s < t$ implies that for all $i\leq p_s$, we have 
$p_t \geq f_{d_i}(p_s)$.  Therefore, $p$ computes a sequence 
of functions $\{g_i : i \in \omega\}$ such that 
for all $i$ and $n$, $g_i(n) \geq f_{d_i}(n)$.  
(Given $i$ and $n$, let $s$ be large enough that $i,n\leq p_s$, 
and output $p_{s+1}$.)  Therefore, $A \oplus p$ computes 
$$\bigoplus_i \Phi_{k(d_i)}(A \oplus g_i) = 
\bigoplus_i H_{d_i}^A = H_{3\cdot 5^e}^A,$$
as was needed.
\end{proof}

We end this section with a question about robustness.  
The formalization of $\BorelDRT^3_2$ given 
above is one we find quite natural.  However, another 
possible way to state the hypothesis of this theorem would 
be 
``Whenever there are Borel codes $T_1,\dots T_\ell$ such that 
for every $X \in k^\omega$, there is an $i$ such that
$X \in \makeset{T_i}$, 
...''

The subtle difference lies in the fact that if 
$X \in \makeset{\cup_{i<\ell} T_i}$, the evaluation map 
for $X$ in that code must also prove that 
$X \in \makeset{T_i}$ or $X \in \makeset{T_i^c}$ for 
each $i < \ell$.  In the slight variant just mentioned, 
it is enough to know that for some $i$, $X \in T_i$ 
(and possibly have no information about $X$ in 
the codes $T_j$ for $j \neq i$.)  This variant 
does not, at least on its face, lead to any conclusion 
about whether, or in what sense, any of 
the $T_i$ must be completely determined.

\begin{question}
How robust is the given formalization of $\BorelDRT^3_2$? 
In particular, is it equivalent to the variant described above?
\end{question}

  \section{Questions}

Several directions of further questions immediately suggest themselves.
Most results here concern $\omega$-models.  It is not immediately 
clear how to formalize the statement ``for every $Z$, there is a 
$\Delta^1_1(Z)$-generic'' in reverse mathematics.  
Once a reasonable reverse mathematics way of 
formalizing these principles is established, it would be natural to ask 
how these principles are related to principles about (completely determined) 
Borel sets.

In the context of $\omega$-models, there are some gaps remaining. 
For example, we have seen that every $\omega$-model of $\DPB$ 
models $\LwCA$ and the existence of $\Delta^1_1$ generics.

\begin{question}
Suppose $M \subseteq 2^\omega$ is closed under join, 
satisfies $\LwCA$, and for every $Z \in M$, there is a 
$G \in M$ that is $\Delta^1_1(Z)$-generic.  Does it follow 
that $\mathcal M \models \DPB$?
\end{question}

One way that the above question could have a negative answer 
would be if $\DPB$ implied some theory of hyperarithmetic 
analysis strictly stronger than $\LwCA$.  

\begin{question}
Which theorems of hyperarithmetic analysis are implied by 
$\DPB$, and which are incomparable with it?
\end{question}

We built an $\omega$-model of $\DPB$ by adjoining 
many mutually $\Sigma^1_1$-generics.

\begin{question}
Does every $\omega$-model of $\DPB$ contain a $\Sigma^1_1$-generic?
\end{question}

Whether in $\omega$-models or full reverse mathematics, many other theorems 
involving Borel sets may now have interesting reverse mathematics content 
when considering their completely determined versions.  We leave
the similar analysis of
``Every completely determined Borel set 
is measurable'' to future work.
We mention that the 
statement ``Every completely determined Borel set has the perfect set property''
is equivalent to $\ATR$, because
``Every closed set has the perfect set property'' already 
implies $\ATR$ by \cite[V.5.5]{sosa}, 
so here the way of defining a Borel set does 
not add additional strength.

Turning now to $\BorelDRT^3_\ell$, we have seen that any $\omega$-model 
of it is closed under hyperarithmetic reduction.  
\begin{question}
Is $\BorelDRT^3_\ell$ a theory of hyperarithmetic analysis?
\end{question}
For any instance of $\BorelDRT^3_\ell$
that is truly well-founded, there is a solution hyperarithmetic in the instance.
However, we do not know anything about the complexity of solutions to 
non-standard instances of $\BorelDRT^3_\ell$.
In particular, we do not know if $\BorelDRT^3_\ell$ holds in $HYP$.

Finally, there is the issue of robustness.  There are some 
possible variations on what could be 
considered as an evaluation map.  For example, a weaker version 
of an evaluation map would be a partial function 
$f:\subseteq T\rightarrow \{0,1\}$ such that $f(\lambda)$ is 
defined; and whenever $\sigma \in T$ is a $\cup$, and $f(\sigma)=1$,
there is an $n$ such that $f(\sigma\concat n) = 1$;
and whenever $\sigma \in T$ is a $\cap$ and $f(\sigma) = 1$, 
for all $n$, $\sigma\concat n \in T$ implies $f(\sigma\concat n) = 1$;
 and similarly for when $f(\sigma) = 0$.  
 Such a partial function has a natural interpretation as a 
 winning strategy in the game in which one player tries to 
 prove that a real is in the given Borel set while 
 another player tries to prove that it is out.  We have 
used the longer name ``completely determined Borel set'' for 
 our notion in order to reserve the term ``determined Borel set'' for 
 this variant.  We did not investigate,
but it would be interesting to know, the extent to which the results 
of this paper are robust under this and other variations on when we 
consider Borel set to be well-defined in reverse mathematics.

\bibliographystyle{alpha}
\bibliography{bib_DPB}

\end{document}